\documentclass[12pt]{amsart}

\usepackage{geometry}
\usepackage{fullpage}
\usepackage{amsmath,amssymb,amsthm,mathtools}
\usepackage{enumerate}
\usepackage{graphicx}
\usepackage{hyperref}
\usepackage{xspace}

\usepackage{tikz}
\usetikzlibrary{positioning}
\usetikzlibrary{arrows.meta}
\usetikzlibrary{backgrounds, calc, fit}
\usetikzlibrary{decorations.pathreplacing}

\theoremstyle{plain}
\newtheorem{theorem}{Theorem}[section]
\newtheorem{proposition}[theorem]{Proposition}
\newtheorem{lemma}[theorem]{Lemma}
\newtheorem{corollary}[theorem]{Corollary}
\newtheorem*{claim}{Claim}  

\theoremstyle{definition}
\newtheorem{definition}[theorem]{Definition}
\newtheorem{remark}[theorem]{Remark}
\newtheorem{example}[theorem]{Example}

\newcommand{\cS}{\ensuremath{\mathcal{S}}\xspace}
\newcommand{\abs}[1]{\ensuremath{\left|#1\right|}}
\newcommand{\littleO}[1]{\ensuremath{o\left(#1\right)}}

\begin{document}

\title{Burning Signed Graphs}

\author{Shaun Fallat}
\address{Department of Mathematics and Statistics, University of Regina, Canada}
\email{shaun.fallat@uregina.ca}
\author{Himanshu Gupta}
\address{Department of Mathematics and Statistics \and Department of Computer Science, University of Regina, Canada}
\email{himanshu.gupta@uregina.ca}
\author{Kamyar Khodamoradi}
\address{Department of Computer Science, University of Regina, Canada}
\email{kamyar.khodamoradi@uregina.ca}
\author{Sandra Zilles}
\address{Department of Computer Science, University of Regina, Canada \and Alberta Machine Intelligence Institute}
\email{sandra.zilles@uregina.ca}

\begin{abstract}
We introduce and analyze a new model of graph burning, in which two competing fires (coloured yellow and green) ignite vertices of a given signed graph and propagate along its edges. The boolean sign of an edge determines whether a fire spreading along the edge changes colour or not. In each step, a player ignites a new vertex in a colour of their choice, while previously activated fires continue to spread. Given a signed graph $\Gamma$, the objective is to burn the maximum possible number of vertices in a single colour; the optimal achievable value is called the \emph{plurality number}\/ of $\Gamma$. We express the plurality number through Hamming distances to a binary code associated with $\Gamma$. In particular, the minimum plurality number over the switching class of $\Gamma$ equals the number of vertices minus the covering radius of this code.

Under certain conditions on the signature, we determine exact values of the plurality number of signed paths. By contrast, we prove hardness results for multiple variants of the problem of determining or approximating the plurality number.
\end{abstract}


\maketitle

\section{Introduction}\label{sec:intro}
Dynamic propagation processes on graphs often lead to parameters that encode how local structures within a graph impact more global graph or network behaviour. One such process is known as {\em graph burning}, which is a discrete-time model introduced to study the spread of influence or contagion across a network. Despite its simple formulation, graph burning gives rise to a rich collection of structural, extremal, and algorithmic problems in graph theory.

In the standard graph burning process \cite{Bonato21,Bonato2016HowToBurn}, all vertices of a finite, simple, undirected graph are initially unburned. At each time step, the fire simultaneously spreads from the previously burned vertices to their unburned neighbours, and one unburned vertex is chosen as a new source and ignited. 
The minimum number of time steps required to burn all vertices of a graph $G$ is called the burning number. This parameter measures the efficiency with which expanding neighbourhoods centred at selected vertices can eventually cover all the vertices in this graph. A prominent conjecture is that the burning number of a connected graph on $n$ vertices never exceeds $\lceil\sqrt{n}\rceil$ \cite{Bonato2016HowToBurn}. Multiple studies explored the computational complexity of calculating the burning number, leading to non-trivial bounds or efficient (exact or approximation) algorithms for specific classes of graphs~\cite{Bessy2018Bounds,Bonato2014Contagion,Bonato2016HowToBurn,Bonato2019Approx} as well as a variety of hardness results~\cite{BessyBJRR17}.
Some model variations involve two players working against one another. For example, Chiarelli et al.~\cite{ChiarelliCJKM26} proposed a setting in which the burning process is slowed down by a \emph{cooling}\/ process. By contrast, in the \emph{builder and arsonist}\/ model, a builder keeps adding new vertices to the graph, while an arsonist aims to burn the graph as quickly as possible~\cite{GundersonKNP25}. 

In this paper, we propose and study a new model of graph burning, motivated by the idea of two competing groups spreading information in social networks. To model situations in which a user (a vertex) either reinforces or counteracts information shared by another user, we use the concept of a signed graph. Each edge between two vertices is signed either positively or negatively, while burned vertices are coloured either yellow or green. Information spreading from a vertex $v$ burned in colour $c$, along an edge $e$, to an uncoloured vertex $v'$ will burn $v'$ in colour $c$ if  $e$ is positive (thus reinforcing the information), and in the opposite colour if $e$ is negative (thus counteracting the information).\footnote{Special rules apply when fires of opposite colour reach an unburned vertex simultaneously; see Section~\ref{sec:model} for a formal definition.} This model gives rise to notions of two-player and one-player games and thus to various mathematically interesting parameters of a graph. Our focus is on a one-player version of the game. As in standard graph burning, starting with only unburned vertices, in each round the already burning fires spread to their neighbours while one new vertex is ignited. However,  the player now has to decide in which colour to ignite the vertex. The largest number of vertices that can be burned in the same colour is the parameter of interest to our study; we dub it the \emph{plurality number}.

Determining the plurality number of a given signed graph is in general a difficult problem. We provide exact values for a subclass of signed paths, while the problem remains open for general classes of graphs with a more complex structure. On the algorithmic side, we prove several hardness results on the problem of computing or approximating the plurality number, including variants with directed graphs and variants in which the burning process is set to terminate after a predefined number of rounds.

Of independent interest is a connection to coding theory, which we reveal through an algebraic formulation of signed graph burning. In particular, over a class $[\Gamma]$ of switching-equivalent signed graphs, the minimal plurality number can be expressed in terms of the covering radius of a binary code that we show to be associated with $[\Gamma]$.

\section{Signed Burning with Two Colours}\label{sec:model}
Let $G$ be a finite simple graph with the vertex set $V(G)$ and the edge set $E(G)$. A \emph{signed graph} $\Gamma = (G, \sigma)$ consists of a graph $G$ together with a \emph{signature} $\sigma \colon E(G) \to \{\pm 1\}$ assigning a positive or negative sign to each edge. 
The graph $G$ is called the \emph{underlying graph} of $\Gamma$. A signed graph whose signature is identically $+1$ or identically $-1$ is denoted $\Gamma_+$ or $\Gamma_-$, respectively; in particular, the unsigned graph $G$ is identified with $\Gamma_+$.
A vertex $u$ is a \emph{positive neighbour} of $v$ if $\sigma(uv) = 1$, and a \emph{negative neighbour} of $v$ if $\sigma(uv) = -1$. The sets of positive and negative neighbours of $v$ are denoted $N^+(v)$ and $N^-(v)$, respectively. The \emph{sign} of a path $u_1 u_2 \cdots u_k$ is the product $\prod_{i=1}^{k-1} \sigma(u_i u_{i+1})$; the path is called \emph{positive} or \emph{negative} according to whether its sign is $+1$ or $-1$. The same terminology applies to cycles. A signed graph is \emph{balanced} if every cycle is positive and \emph{unbalanced} otherwise.

Signed graphs were introduced by Harary \cite{Harary1953}. Originally used to model structural balance in social groups \cite{CartwrightHarary1956}, they have since found applications in many areas. In statistical physics, unbalanced cycles correspond to frustration in spin glass models \cite{Toulouse1977}. In data science, signed graphs underlie correlation clustering \cite{BBC2004} and the study of social networks with positive and negative ties \cite{TCAL2016}. They also have a rich spectral theory \cite{BCKW2018}. We refer to \cite{Zaslavsky1982} for further reading.

Given a signed graph $\Gamma = (G, \sigma)$ and a subset $U \subseteq V(G)$, the \emph{switching} of $\Gamma$ by $U$ is the resulting signed graph $\Gamma^U = (G, \sigma^U)$, where
\[
\sigma^U(uv) =
\begin{cases}
\phantom{-}\sigma(uv), &\text{if } u, v \in U \text{ or } u, v \notin U,\\
-\sigma(uv), &\text{otherwise}.
\end{cases}
\]
Two signed graphs $\Gamma = (G, \sigma)$ and $\Gamma' = (G, \sigma')$ are called \emph{switching equivalent}, written $\Gamma \sim \Gamma'$, if there exists a subset $U \subseteq V(G)$ such that $\Gamma' = \Gamma^U$. Switching equivalence is an equivalence relation on the set of all signatures of $G$. A well-known result characterizes the property of being balanced in terms of switching: a signed graph $\Gamma$ is balanced if and only if it is switching equivalent to $\Gamma_+$
(see \cite{Harary1953} and \cite[Corollary~3.3]{Zaslavsky1982}). Since a forest contains no cycles, it is evident that every signed forest is balanced and therefore is switching equivalent to $\Gamma_+$.

We now define a burning process on a signed graph in which the fire carries one of two colours, yellow and green.\footnote{The choice of these two colours is a nod to the University of Regina.}

\begin{definition}[Signed Graph Burning Process]\label{defn:signed_burning}
Let $\Gamma = (G, \sigma)$ be a signed graph. At any given time, each vertex is in exactly one of three states: unburned, burned yellow, or burned green. Initially ($t=0$) every vertex is unburned; once a vertex is burned (yellow or green), it remains in that state. For each time $t\geq 1$, the following two steps are performed in order.
\begin{enumerate}
    \item \textbf{Spread:} For each unburned vertex $v$, compute the yellow influence
    $$s_t^{y}(v) = |\{u \in N^{+}(v) : u \text{ is yellow}\}| + |\{u \in N^{-}(v) : u \text{ is green}\}|,$$
    and the green influence
    $$s_t^{g}(v) = |\{u \in N^{+}(v) : u \text{ is green}\}| + |\{u \in N^{-}(v) : u \text{ is yellow}\}|.$$
 Then update simultaneously every unburned vertex $v$ as follows:
    \begin{enumerate}
        \item burn $v$ yellow if $s_t^{y}(v) > s_t^{g}(v)$,
        \item burn $v$ green if $s_t^{y}(v) < s_t^{g}(v)$,
        \item leave $v$ unburned if $s_t^{y}(v) = s_t^{g}(v)$.
    \end{enumerate}
    \item \textbf{Ignition:} Pick an unburned vertex, if one exists, and burn it yellow or green.
\end{enumerate}
This process terminates when every vertex is either burned yellow or burned green.
\end{definition}

\begin{remark}[Arithmetic form of the process]\label{rem:arithmetic}
Identify yellow with $1$, green with $-1$, unburned with $0$, and let $\nu_t \colon V(G) \to \{0,1,-1\}$ record the state of every vertex at the end of time $t$, so that $\nu_0 \equiv 0$. A burned neighbour $u$ of $v$ contributes to $s_t^{y}(v)$ iff $\sigma(vu)\nu_{t-1}(u) = 1$ and to $s_t^{g}(v)$ iff $\sigma(vu)\nu_{t-1}(u) = -1$, while an unburned neighbour contributes
$\sigma(vu)\nu_{t-1}(u) = 0$ to neither. Hence, writing
\[
  h_t(v) \; :=\; \sum\nolimits_{u \in N(v)} \sigma(vu)\,\nu_{t-1}(u),
\]
we have $h_t(v) = s_t^{y}(v) - s_t^{g}(v)$ for unburned $v$, and the spread step results in 
\[
  \nu_t(v) \;=\;
  \begin{cases}
    \nu_{t-1}(v), & \text{if } \nu_{t-1}(v) \neq 0,\\[2pt]
    \operatorname{sgn} (h_t(v)), & \text{if } \nu_{t-1}(v) = 0,
  \end{cases}
\]
with 
$\operatorname{sgn} (0) := 0$. The ignition step then changes a single remaining $0$ to either $1$ or $-1$, producing $\nu_t$. The process ends once $\nu_T \in \{\pm 1\}^{V(G)}$ for some~$T$.
\end{remark}

A complete signed burning process is encoded by a \emph{burning sequence} $$\mathcal{I} = \bigl(v_1^{c_1},\, v_2^{c_2},\, \ldots,\, v_k^{c_k}\bigr),$$ where $v_t \in V(G)$ is the vertex ignited at time $t$, $c_t \in \{y, g\}$ is its colour, and every vertex is burned by the end of the process. The \emph{termination time} $T=T(\mathcal{I})$ is the first time at the end of which every vertex is burned. The spread step precedes the ignition step, so the final round may have no vertex left to ignite. Hence $k=|\mathcal{I}|$ equals $T$ or $T-1$. 

For $0\leq t \leq T$, let $\nu_{\mathcal{I},t}$ be the \emph{configuration} of $\mathcal{I}$ at time $t$, that is, the function $\nu_t$ of Remark~\ref{rem:arithmetic} determined by $\mathcal{I}$. Let $Y_{\mathcal{I},t}=\{v:\nu_{\mathcal{I},t}(v)=1\}$ and $G_{\mathcal{I},t}=\{v:\nu_{\mathcal{I},t}(v)=-1\}$ be the sets of vertices burned yellow and burned green at the end of time $t$. When $t = T(\mathcal{I})$ is clear from context we drop the second subscript and write  $\nu_{\mathcal{I}}$, $\mathcal{Y}_{\mathcal{I}}$ and $\mathcal{G}_{\mathcal{I}}$. Thus $\nu_{\mathcal{I}} \in \{\pm 1\}^{V(G)}$, and we call it the \emph{final colouring} of $\mathcal{I}$; $\{\mathcal{Y}_{\mathcal{I}}, \mathcal{G}_{\mathcal{I}}\}$ is a partition of $V(G)$, possibly with one part empty. We define  $\nu_{\mathcal{I},t} = \nu_{\mathcal{I}}$, $\mathcal{Y}_{\mathcal{I},t}=\mathcal{Y}_{\mathcal{I}}$ and $\mathcal{G}_{\mathcal{I},t} = \mathcal{G}_{\mathcal{I}}$ for all $T<t\leq |V(G)|$. Let $\mathcal{B}(\Gamma)$ denote the set of all burning sequences of $\Gamma$.  We define $\mathcal{R}(\Gamma) \;:=\; \bigl\{\, \nu_{\mathcal{I}} \ :\
\mathcal{I} \in \mathcal{B}(\Gamma) \,\bigr\} \;\subseteq\; \{\pm 1\}^{V(G)}$ to be the set of \emph{reachable colourings} of $\Gamma$.

For $U \subseteq V(G)$ define $\Phi_U \in \{\pm 1\}^{V(G)}$ by $\Phi_U(v) = -1$ if $v \in U$ and $\Phi_U(v) = 1$ otherwise. So the switched signature $\sigma^U$ satisfies
\begin{equation}\label{eqn:switch_product}
  \sigma^U(vu) \;=\; \Phi_U(v)\,\Phi_U(u)\,\sigma(vu).
\end{equation}
Products of vectors below are taken coordinate-wise, and we define $\Phi_U \cdot \mathcal{M} := \{\Phi_U \cdot \nu : \nu \in \mathcal{M}\}$ for a set $\mathcal{M} \subseteq \{\pm 1\}^{V(G)}$.

Let $\mathcal{I}=(v_1^{c_1},\ldots,v_k^{c_k})$ be a burning sequence of $\Gamma$, and let $U\subseteq V(G)$. The \emph{switched burning sequence}
$\mathcal{I}^U=(v_1^{c'_1},\ldots,v_k^{c'_k})$ ignites the same vertices. Its colours are $c'_t=\bar c_t$ if $v_t\in U$, and $c'_t=c_t$ if $v_t\notin U$. Here $\bar{y} = g$ and $\bar{g} = y$. Equivalently, $c_t' = \Phi_U(v_t)\,c_t$. 

The following useful lemma is proven by induction on $t$.

\begin{lemma}\label{lem:switching_burning}
Let $\Gamma = (G,\sigma)$ be a signed graph, $U \subseteq V(G)$, and $\mathcal{I} \in \mathcal{B}(\Gamma)$. Then $\mathcal{I}^U \in \mathcal{B}(\Gamma^U)$, $T(\mathcal{I}^U) = T(\mathcal{I})$, and  for all $0 \leq t \leq T = T(\mathcal{I})$,
\begin{equation}\label{eqn:switching_set_identities}
  \nu_{\mathcal{I}^U\!,\,t} \;=\; \Phi_U \cdot \nu_{\mathcal{I},t}.
\end{equation}
Equivalently,
$
\mathcal{Y}_{\mathcal{I}^U\!,\,t} = \bigl(\mathcal{Y}_{\mathcal{I},t} \setminus U\bigr) \cup \bigl(\mathcal{G}_{\mathcal{I},t} \cap U\bigr)$, and
$
\mathcal{G}_{\mathcal{I}^U\!,\,t} = \bigl(\mathcal{G}_{\mathcal{I},t} \setminus U\bigr) \cup \bigl(\mathcal{Y}_{\mathcal{I},t} \cap U\bigr).
$
\end{lemma}

\begin{proof}
Write $\nu_t = \nu_{\mathcal{I},t}$ for the process $\mathcal{I}$ on $\Gamma$ and
$\nu_t' = \nu_{\mathcal{I}^U\!,\,t}$ for the process $\mathcal{I}^U$ on $\Gamma^U$, and let
$h_t$ and $h_t'$ be the corresponding influence sums from Remark~\ref{rem:arithmetic}. We prove
\eqref{eqn:switching_set_identities} by induction on $t$, the case $t = 0$ being clear since
$\nu_0 = \nu_0' \equiv 0$. Assume $\nu_{t-1}' = \Phi_U \cdot \nu_{t-1}$ for some $t \geq 1$. By
\eqref{eqn:switch_product} and $\Phi_U(u)^2 = 1$,
\begin{align*}  
h_t'(v) &= \sum_{u \in N(v)} \sigma^U(vu)\,\nu_{t-1}'(u) = \sum_{u \in N(v)} \Phi_U(v)\Phi_U(u)\sigma(vu)\,\Phi_U(u)\nu_{t-1}(u)\\
&= \Phi_U(v)\,h_t(v).
\end{align*}

A vertex is unburned in one process exactly when it is unburned in the other, since
$\nu_{t-1}'(v) = 0$ if and only if $\nu_{t-1}(v) = 0$. If $\nu_{t-1}(v) \neq 0$ then both
processes keep the colour of $v$, so $\nu_t'(v) = \nu_{t-1}'(v) = \Phi_U(v)\nu_{t-1}(v) = \Phi_U(v)\nu_t(v)$ by the induction hypothesis. If $\nu_{t-1}(v) = 0$ the spread step gives $\nu_t'(v) = \operatorname{sgn} (h_t'(v)) = \Phi_U(v)\operatorname{sgn} (h_t(v)) = \Phi_U(v)\nu_t(v)$, because $\operatorname{sgn}$ is an odd function. In particular, the two processes have the same unburned set after the spread step, so a vertex is available for ignition in one precisely when it is available in the other. If an ignition occurs at time $t$ it is at the same vertex $v_t$ in both processes, receiving the colour $c_t$ in $\mathcal{I}$ and $\Phi_U(v_t)c_t$ in
$\mathcal{I}^U$, which again agrees with \eqref{eqn:switching_set_identities}. This completes the induction. Consequently the two processes burn their last vertex at the same time and ignite at the same
rounds, so $T(\mathcal{I}^U) = T(\mathcal{I})$ and
$\mathcal{I}^U \in \mathcal{B}(\Gamma^U)$. Moreover, the equivalence follows immediately.
\end{proof}

\begin{corollary}\label{cor:R_switching}
For every signed graph $\Gamma$ and every $U \subseteq V(G)$ we have
$\mathcal{R}(\Gamma^U) = \Phi_U \cdot \mathcal{R}(\Gamma)$. In particular
$\mathcal{R}(\Gamma) = -\mathcal{R}(\Gamma)$.
\end{corollary}
\begin{proof}
The map $\mathcal{I} \mapsto \mathcal{I}^U$ is a bijection $\mathcal{B}(\Gamma) \to \mathcal{B}(\Gamma^U)$, since $(\mathcal{I}^U)^U = \mathcal{I}$ and Lemma~\ref{lem:switching_burning} applies for both implications; the first claim is now \eqref{eqn:switching_set_identities} at $t = T$. Taking $U = V(G)$ gives $\Phi_U \equiv -1$, and $\Gamma^{V(G)} = \Gamma$ by \eqref{eqn:switch_product}, whence $\mathcal{R}(\Gamma) = -\mathcal{R}(\Gamma)$.
\end{proof}

\begin{definition}[Signed Burning Parameters]\label{defn:parameters}
Let $\Gamma = (G,\sigma)$ be a signed graph. Then the signed burning number is $b(\Gamma)= \min_{\mathcal{I}\,\in\,\mathcal{B}(\Gamma)} T(\mathcal{I})$ and the plurality number is $Y(\Gamma) = \max_{\mathcal{I}\,\in\,\mathcal{B}(\Gamma)}\left(\max\bigl\{|\mathcal{Y}_{\mathcal{I}}|,\, |\mathcal{G}_{\mathcal{I}}|\bigr\}\right)$.
\end{definition}

With the aid of Corollary~\ref{cor:R_switching}, we have the following lemma.

\begin{lemma}\label{lem:symmetry}
For any signed graph $\Gamma$ on $n$ vertices,
\[
    Y(\Gamma)
    \;=\; \max_{\mathcal{I}\,\in\,\mathcal{B}(\Gamma)} |\mathcal{Y}_{\mathcal{I}}|
    \;=\; \max_{\mathcal{I}\,\in\,\mathcal{B}(\Gamma)} |\mathcal{G}_{\mathcal{I}}|
    \;\geq\; \lceil n/2 \rceil.
\]
\end{lemma}
\begin{proof}
By Corollary~\ref{cor:R_switching}, $\nu \mapsto -\nu$ is a bijection of $\mathcal{R}(\Gamma)$ onto itself,  exchanging $\{v : \nu(v) = 1\}$ with $\{v : \nu(v) = -1\}$. Hence $\max_{\mathcal{I}} |\mathcal{Y}_{\mathcal{I}}| = \max_{\mathcal{I}} |\mathcal{G}_{\mathcal{I}}|=Y(\Gamma)$.
The bound $Y(\Gamma) \geq \lceil n/2 \rceil$ follows since $|\mathcal{Y}_{\mathcal{I}}| + |\mathcal{G}_{\mathcal{I}}| = n$ for every $\mathcal{I} \in \mathcal{B}(\Gamma)$.
\end{proof}

By Lemma~\ref{lem:symmetry}, we adopt $Y(\Gamma) = \max_{\mathcal{I}\,\in\,\mathcal{B}(\Gamma)} |\mathcal{Y}_{\mathcal{I}}| 
\ = \max_{\nu \in \mathcal{R}(\Gamma)} |\{v : \nu(v) = 1\}|$; that is, the plurality number is the maximum yellow count over all burning sequences of $\Gamma$. In standard graph burning \cite{Bonato2014Contagion,Bonato21} (see also  \cite{Bessy2018Bounds,Bonato2016HowToBurn,Bonato2019Approx} for bounds and approximation algorithms for the burning number; and see \cite{Bonato2014Contagion,GarciaDiaz2022,Pereira2024} for applications) on an unsigned graph, at each time step every unburned vertex adjacent to a previously burned vertex becomes burned, and then an unburned vertex is ignited, if one exists. The (unsigned) burning number $b(G)$ is the minimum number of steps required to burn every vertex of $G$. 

The following proposition relates the signed to the unsigned burning number (statements (ii) and (iii)); it is proven by inductively analyzing optimal burning sequences. Statement (i), which shows that switching-equivalent graphs have the same signed burning number, is a straightforward consequence of Lemma~\ref{lem:switching_burning}.

\begin{proposition}\label{prop:burning_bound}
Let $G$ be a graph, and let $\Gamma = (G, \sigma)$ be a signed graph.
\begin{enumerate}[{\upshape(i)}]
    \item If $\Gamma' = (G, \sigma')$ is a signed graph with $\Gamma' \sim \Gamma$, then $b(\Gamma') = b(\Gamma)$.
    \item The (unsigned) burning number satisfies $b(G) \leq b(\Gamma)$.
    \item If $\Gamma$ is a balanced signed graph, then $b(G) = b(\Gamma)$. In particular, if $G$ is a forest, then $b(G) = b(\Gamma)$.
\end{enumerate}
\end{proposition}
\begin{proof}
\begin{enumerate}[{\upshape(i)}]
\item We have $\Gamma' = \Gamma^U$ for some $U \subseteq V(G)$ since $\Gamma' \sim \Gamma$. Consider a burning sequence $\mathcal{I} \in \mathcal{B}(\Gamma)$ with $T(\mathcal{I}) = b(\Gamma)$. Lemma~\ref{lem:switching_burning} implies that $\mathcal{I}^U \in \mathcal{B}(\Gamma')$ with $T(\mathcal{I}^U) = T(\mathcal{I}) = b(\Gamma)$, so $b(\Gamma') \leq b(\Gamma)$. The reverse inequality follows by symmetry. Therefore, $b(\Gamma') = b(\Gamma)$.

\item Fix an optimal burning sequence $\mathcal{I} = (v_1^{c_1}, \ldots, v_k^{c_k})$ for $\Gamma$. So $T = T(\mathcal{I}) = b(\Gamma)$, and write $\nu_t = \nu_{\mathcal{I},t}$ and $h_t$ for the associated configurations and influence sums. Run standard burning of $G$ with the ignitions $(v_1,\ldots,v_k)$. When some $v_t$ is already burned, ignite an arbitrary unburned vertex instead. We show that this process burns every vertex by time $T$. Let $Z_t$ denote the set of burned vertices after time $t$ in that standard process, and let $W_t = \mathcal{Y}_{\mathcal{I},t} \cup \mathcal{G}_{\mathcal{I},t}$ denote the set of all burned vertices in the signed burning process after time $t$. We claim $W_t \subseteq Z_t$ for all $t$, by induction on $t$. The base case $t = 0$ is immediate. Suppose $W_{t-1} \subseteq Z_{t-1}$. Let $v$ burn during the spread step of the signed process at time $t$. Then $h_t(v)\neq0$, so some term $\sigma(uv)\nu_{t-1}(u)$ is nonzero. Thus
$v$ has a burned neighbour in $W_{t-1}\subseteq Z_{t-1}$. Standard burning at time $t$ burns every unburned neighbour of $Z_{t-1}$, so $v\in Z_t$. At time $t \leq k$ the signed process ignites $v_t$. The standard process ignites $v_t$ as well if it is still unburned, and otherwise ignites any unburned vertex. In both cases $v_t \in Z_t$, since $v_t \in Z_{t-1}$ whenever it is already burned. Hence $W_t \subseteq Z_t$. Since $W_T = V(G)$, we get $Z_T = V(G)$ and $b(G) \leq b(\Gamma)$.

\item If $\Gamma$ is balanced then $\Gamma \sim \Gamma_+$ and by (i), $b(\Gamma) = b(\Gamma_+)$. By (ii), $b(G) \leq b(\Gamma_+)$. Fix an optimal burning sequence $(v_1, \ldots, v_k)$ for $G$, so that standard burning terminates at time $T = b(G)$. Run the signed burning process on $\Gamma_+$ with ignition sequence $(v_1^y, \ldots, v_k^y)$. Since all edges of $\Gamma_+$ are positive, an easy induction on $t$ shows that every configuration of this process is nonnegative: the ignitions are yellow, and $\operatorname{sgn} (h_t) \geq 0$ whenever $\nu_{t-1} \geq 0$. Hence $h_t(v)$ is the number of burned neighbours of $v$ and no green-burned vertex ever arises; the spreading rule thus reduces to that of standard graph burning. Hence $b(\Gamma_+) \leq b(G)$, giving $b(\Gamma_+) = b(G)$. Consequently, $b(G) = b(\Gamma)$. A forest has no cycles, so any signed forest is balanced. \qedhere
\end{enumerate}
\end{proof}

\begin{example}\label{ex:strict}
Let $\Gamma = (G,\sigma)$ be the signed graph in Figure~\ref{fig:strict_burning_number_example}, in which the dashed edge is negative and every other edge is positive. Its unique cycle is negative, so $\Gamma$ is unbalanced. By computer verification, $b(G) = 3$ and $b(\Gamma) = 4$. Hence the inequality in Proposition~\ref{prop:burning_bound}\,(ii) can be strict.
\begin{figure}[ht]
    \centering
    \begin{tikzpicture}[scale=0.5,
  vtx/.style={circle,draw,fill=black,inner sep=1.6pt},
  lab/.style={font=\small,inner sep=2pt},
  posedge/.style={thick}, negedge/.style={thick,dashed}]
  \node[vtx,label={[lab]above:}]        (v1)  at (0,1) {};
  \node[vtx,label={[lab]above:}]        (v2)  at (1.4,1) {};
  \node[vtx,label={[lab]below:}]        (v3)  at (1.4,0) {};
  \node[vtx,label={[lab]below:}]        (v4)  at (0,0) {};
  \node[vtx,label={[lab]above right:}]  (v5)  at (-1.2,1) {};
  \node[vtx,label={[lab]above:}]        (v6)  at (-2.4,1) {};
  \node[vtx,label={[lab]right:}]        (v8)  at (2.8,0) {};
  \node[vtx,label={[lab]above left:}]   (v7)  at (-1.2,0) {};
  \node[vtx,label={[lab]above left:}]   (v9)  at (-2.4,0) {};
  \node[vtx,label={[lab]below:}]     (v10) at (-3.6,0) {};
  \draw[posedge] (v1)--(v2); \draw[posedge] (v2)--(v3); \draw[posedge] (v1)--(v4);
  \draw[negedge] (v3)--(v4);
  \draw[posedge] (v1)--(v5); \draw[posedge] (v5)--(v6);
  \draw[posedge] (v3)--(v8);
  \draw[posedge] (v4)--(v7); \draw[posedge] (v7)--(v9); \draw[posedge] (v9)--(v10);
\end{tikzpicture}
    \caption{An unbalanced signed graph $\Gamma = (G,\sigma)$ with $3=b(G)< b(\Gamma) = 4$.}
    \label{fig:strict_burning_number_example}
\end{figure}
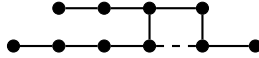
\end{example}

\begin{proposition}\label{prop:switching_distance}
Let $\Gamma$ and $\Gamma'$ be signed graphs on the same underlying graph $G$ on $n$ vertices, and suppose $\Gamma' = \Gamma^{\,U}$ for some $U \subseteq V(G)$. Then
\[
  Y(\Gamma') = \max\nolimits_{\,\nu \,\in\, \mathcal{R}(\Gamma)} \bigl|\{v : \nu(v) = \Phi_U(v)\}\bigr|
  = n - \min\nolimits_{\,\nu \,\in\, \mathcal{R}(\Gamma)} d_H\bigl(\nu, \Phi_U\bigr),
\]
where $d_H(\nu,\nu') = |\{v : \nu(v) \neq \nu'(v)\}|$ denotes the Hamming distance.
\end{proposition}

\begin{proof}
By Lemma~\ref{lem:switching_burning} and Corollary~\ref{cor:R_switching} the map $\nu \mapsto \mu:=\Phi_U \cdot \nu$ is a bijection $\mathcal{R}(\Gamma) \to \mathcal{R}(\Gamma')$, and $\mu(v)=\Phi_U(v)\nu(v) = 1$ if and only if $\nu(v) = \Phi_U(v)$. Hence
\[
  Y(\Gamma') = \max\nolimits_{\mu \in \mathcal{R}(\Gamma')} \bigl|\{v : \mu(v) = 1\}\bigr|
  = \max\nolimits_{\nu \in \mathcal{R}(\Gamma)} \bigl|\{v : \nu(v) = \Phi_U(v)\}\bigr|,
\]
since $\nu$ and $\Phi_U$ agree at exactly $n - d_H(\nu,\Phi_U)$ vertices.
\end{proof}

\begin{remark}\label{rem:covering_radius}
(i) Identify $\{\pm1\}$ with the field $\mathbb{F}_2$ via $x\mapsto(1-x)/2$. Then $R(\Gamma)$ becomes a binary code of length $n$. Proposition~\ref{prop:switching_distance} says
that $n-Y(\Gamma^U)$ is the distance from the word $\Phi_U$ to this code. As $U$ ranges over all subsets of $V(G)$, the vector $\Phi_U$ ranges over all of $\{\pm 1\}^{V(G)}$ and the signed graph $\Gamma^{\,U}$ ranges over the whole switching class of $\Gamma$, so
    \[
    \min\nolimits_{\Gamma' \sim \Gamma} Y(\Gamma')
    = n - \max\nolimits_{U \subseteq V(G)} \ \min\nolimits_{\nu \in \mathcal{R}(\Gamma)} d_H\bigl(\nu, \Phi_U\bigr)
    = n - \operatorname{cr}\bigl(\mathcal{R}(\Gamma)\bigr),
    \]
    where $\operatorname{cr}(C) = \max_{\Phi \in \{\pm 1\}^{V(G)}} \min_{\nu \in C} d_H(\nu,\Phi)$ denotes the covering radius of a code $C \subseteq \{\pm 1\}^{V(G)}$. Determining the signature that minimises the plurality number within a switching class is thus a covering radius problem for $\mathcal{R}(\Gamma)$.
    
    (ii) Taking $\Gamma' = \Gamma$ and $U = \emptyset$ recovers the definition of $Y(\Gamma)$, so the content of Proposition~\ref{prop:switching_distance} lies entirely in changing the representative of the switching class. By Corollary~\ref{cor:R_switching} the set $\mathcal{R}$ is determined by that class up to coordinate-wise multiplication by a sign vector, so the choice of representative is immaterial.
\end{remark}

\section{Plurality Number of Signed Path Graphs}\label{sec:paths}

Let $\Gamma=(P_n,\sigma)$ be a signed path of order $n\geq 2$ with $V(P_n)=\{v_1,\ldots,v_n\}$ and $E(P_n)=\{v_iv_{i+1}: 1\le i\le n-1\}$. Assume that the vertices are ordered from left to right. We write $\sigma_i=\sigma(v_iv_{i+1})$ for the sign of the edge $v_iv_{i+1}$. It is well known that $b(P_n)=\lceil\sqrt{n}\,\rceil$ (see \cite[Theorem~2.9]{Bonato2016HowToBurn}). Since every
signed path is balanced, Proposition~\ref{prop:burning_bound}\,(iii) gives $b(\Gamma)=\lceil \sqrt{n}\,\rceil$ for every signature $\sigma$, so the signed burning number
of a path is already determined by its order. We therefore devote this section to the plurality number of signed paths.

A \emph{negative segment} of $\Gamma$ is a subpath of $P_n$ with at least one edge, all of whose edges are negative, and which is maximal with respect to these two properties. The negative segments are pairwise vertex disjoint. Consider the subgraph of $P_n$ induced by the vertices that lie in no negative segment. Its components are the \emph{positive segments} of $\Gamma$. The segments therefore partition $V(P_n)$ into subpaths, which we list from left to right as $S_1,S_2,\ldots,S_r$. The \emph{order} of a segment is its number of vertices. When $S\neq S_1$, we write $S'$ for the segment preceding $S$.

Three immediate properties will be used without comment. Every edge of a negative segment is negative and every edge of a positive segment is positive, since a negative edge lying in no negative segment would contradict maximality. Two consecutive segments are never both positive. Indeed, two adjacent
vertices that lie in no negative segment lie in the same component. So one of any two consecutive segments is negative, and the edge joining them is positive. Finally, negative segments have order at least $2$, while positive segments may have order~$1$.

From now on $m_1,\ldots,m_s$ denote the orders of the negative segments of
$\Gamma$, listed from left to right. For a segment $S$ put $e(S)=|\{v_1,v_n\}\cap V(S)|$, and let $e_j=e(S)$, where $S$ is the $j$-th negative segment. Notice that $e_j=0$ unless this segment is $S_1$ or $S_r$,
and $e_j=2$ only when $\Gamma=\Gamma_-$.

Theorems~\ref{thm:path-lower} and ~\ref{thm:upper} provide lower and upper bounds on the plurality number of signed paths, respectively. They are proved in Sections~\ref{sec:lower} and \ref{sec:upper}.

\begin{theorem}\label{thm:path-lower}
Let $\Gamma=(P_n,\sigma)$ be a signed path with segments $S_1,\ldots,S_r$, and let $m_1,\ldots,m_s$ be the orders of its negative segments. Set $\varepsilon=1$ if $S_1$ and $S_r$ are negative segments whose orders are congruent to $2$ modulo $3$, and $\varepsilon=0$ otherwise. Then
\[
  Y(\Gamma)\;\ge\; n-\sum_{i=1}^{s}\Big\lfloor \frac{m_i}{3}\Big\rfloor-\varepsilon .
\]
\end{theorem}

\begin{theorem}
\label{thm:upper}
Let $\Gamma=(P_n,\sigma)$ be a signed path with segments $S_1,\ldots,S_r$, and let $m_1,\ldots,m_s$ be the orders of its negative segments. Then
\[
  Y(\Gamma)\le n-\sum_{i=1}^{s}\Big\lfloor\frac{m_i-1+e_i}{3}\Big\rfloor .
\]
\end{theorem}

Combining Theorems~\ref{thm:path-lower} and \ref{thm:upper}, a case distinction on the number of negative segments proves the following result determining $Y(\Gamma)$ for certain signed paths:

\begin{theorem}\label{thm:exact}
Let $\Gamma=(P_n,\sigma)$ be a signed path with segments $S_1,\ldots,S_r$, let $m_1,\ldots,m_s$ be the orders of its negative segments, and let $\varepsilon$ be as in Theorem~\ref{thm:path-lower}. Then 
\[
  Y(\Gamma)=n-\sum_{i=1}^{s}\Big\lfloor \frac{m_i}{3}\Big\rfloor-\varepsilon = n-\sum_{i=1}^{s}\Big\lfloor\frac{m_i-1+e_i}{3}\Big\rfloor
\]
if and only if either $\Gamma=\Gamma_-$, or both of the following hold:
\begin{enumerate}[{\upshape(i)}]
\item every negative segment other than $S_1$ and $S_r$ has order not divisible by $3$;
\item $\varepsilon=0$.
\end{enumerate}
In particular, $Y(\Gamma_-) = \Big\lfloor \frac{2n+1}{3}\Big\rfloor$.
\end{theorem}
\begin{proof}
Write $\mathrm{LB}$ and $\mathrm{UB}$ for the bounds in Theorems~\ref{thm:path-lower} and~\ref{thm:upper}, respectively. Since $\mathrm{LB}\le Y(\Gamma)\le\mathrm{UB}$, the two equalities in the statement hold exactly when $\mathrm{LB}=\mathrm{UB}$. Put $\delta_i=\lfloor m_i/3\rfloor-\lfloor (m_i-1+e_i)/3\rfloor$, so that
\[
  \mathrm{UB}-\mathrm{LB}=\sum_{i=1}^{s}\delta_i+\varepsilon .
\]
Checking the residues of $m_i$ modulo $3$ gives
\[
  \delta_i=\begin{cases}
    1 & \text{if } e_i=0 \text{ and } m_i\equiv 0 \, (\bmod \, 3),\\
   -1 & \text{if } e_i=2 \text{ and } m_i\equiv 2 \, (\bmod \, 3),\\
    0 & \text{otherwise.}
  \end{cases}
\]
If $s=0$, then $\Gamma=\Gamma_+$ and $\mathrm{LB}=\mathrm{UB}=n$, and (i) and (ii) hold trivially. If $s\ge 1$ and $r=1$, then $\Gamma=\Gamma_-$. The only negative segment has $m_1=n$ and $e_1=2$. Moreover $\varepsilon=1$ exactly when $n\equiv 2\, (\bmod \, 3)$, that is, exactly when $\delta_1=-1$. So $\mathrm{LB}=\mathrm{UB}$. If $s\geq 1$ and $r\ge 2$, then every $e_i$ is at most $1$. Then every $\delta_i$ is nonnegative, and so is $\varepsilon$. Hence $\mathrm{LB}=\mathrm{UB}$ if and only if every $\delta_i$ vanishes and $\varepsilon=0$. Since $e_i=0$ exactly for the negative segments other than $S_1$ and $S_r$, this is the conjunction of (i) and (ii). For $\Gamma=\Gamma_-$ we obtain $Y(\Gamma_-)=n-\lfloor (n+1)/3\rfloor$, which equals $\lfloor (2n+1)/3\rfloor$ in each of the cases $n\equiv 0,1,2\, (\bmod \, 3)$.
\end{proof}

Theorems~\ref{thm:path-lower} and \ref{thm:exact} yield the following corollary.

\begin{corollary}\label{cor:minimum}
Let $\Gamma=(P_n,\sigma)$ be a signed path. Then
$Y(\Gamma)\;\ge\;Y(\Gamma_-)\;=\;\big\lfloor\frac{2n+1}{3}\big\rfloor$. If $n\equiv 2\, (\bmod \, 3)$, then equality holds if and only if $\Gamma=\Gamma_-$.
\end{corollary}
\begin{proof}
By Theorem~\ref{thm:exact} we have $Y(\Gamma_-)=n-\lfloor (n+1)/3\rfloor$, so by
Theorem~\ref{thm:path-lower} it suffices to show that
$\sum_{i=1}^{s}\lfloor m_i/3\rfloor+\varepsilon\le\lfloor (n+1)/3\rfloor$.
The negative segments are pairwise vertex disjoint, so $\sum_{i=1}^{s} m_i\le n$.
If $\varepsilon=0$, then $\sum_{i=1}^{s}\lfloor m_i/3\rfloor\le\lfloor\sum_{i=1}^{s} m_i/3\rfloor\le\lfloor n/3\rfloor$.
If $\varepsilon=1$ and $s=1$, then $\Gamma=\Gamma_-$ and the two sides are equal.
If $\varepsilon=1$ and $s\ge 2$, then $S_1$ and $S_r$ are distinct negative segments with
$m_1\equiv m_s\equiv 2\, (\bmod \, 3)$, whence
\[
  \sum_{i=1}^{s}\Big\lfloor\frac{m_i}{3}\Big\rfloor+\varepsilon
  \;\le\;\frac{m_1-2}{3}+\frac{m_s-2}{3}+\sum_{i\neq 1,s}\frac{m_i}{3}+1
  \;\le\;\frac{n-1}{3}.
\]
Since the left side is an integer, it is at most $\lfloor (n-1)/3\rfloor$, which proves the inequality. Now let $n\equiv 2\, (\bmod \, 3)$ and $\Gamma\neq\Gamma_-$. Then
$\lfloor n/3\rfloor=\lfloor (n-1)/3\rfloor=(n-2)/3=\lfloor (n+1)/3\rfloor-1$, so the first and
the third case give $\sum_{i=1}^{s}\lfloor m_i/3\rfloor+\varepsilon\le\lfloor (n+1)/3\rfloor-1$.
Theorem~\ref{thm:path-lower} now gives $Y(\Gamma)\ge Y(\Gamma_-)+1$.
\end{proof}

\begin{remark}
Every signed path is balanced, so the signatures of $P_n$ form a single switching class.
By Remark~\ref{rem:covering_radius}(i),
the minimum of $Y$ over that class is $n-\operatorname{cr}(\mathcal{R}(\Gamma_+))$.
Corollary~\ref{cor:minimum} therefore computes the covering radius of the code of reachable
colourings of the path,
\[
  \operatorname{cr}\bigl(\mathcal{R}(\Gamma_+)\bigr)=\Big\lfloor\frac{n+1}{3}\Big\rfloor ,
\]
and the alternating vector $\Phi^\ast$ given by $\Phi^\ast(v_i)=(-1)^{i+1}$, which is the one belonging to $\Gamma_-$, is a deep hole of this code.
For $n\equiv 2\, (\bmod \, 3)$ it is the only deep hole, up to a global change of sign.
\end{remark}

Both bounds are proved on the all positive path, as follows. Let $U$ be the set of vertices $v_i$ for which the path $v_1v_2\cdots v_i$ is negative, where $2\leq i\leq n$. Note that the paths $v_1\cdots v_j$ and $v_1\cdots v_{j+1}$ have different signs exactly when $\sigma_j=-1$, so $\Phi_U(v_j)\Phi_U(v_{j+1})=\sigma_j$ for every $1\leq j\leq n-1$. By \eqref{eqn:switch_product} we have $\Gamma=(\Gamma_+)^U$. In particular, $\Phi_U$ alternates along every negative segment, it is constant on every positive segment, and it takes the same value at the two ends of the edge joining two consecutive segments. Proposition~\ref{prop:switching_distance}, applied to $\Gamma_+$ and $(\Gamma_+)^U=\Gamma$, gives
\begin{equation}\label{eqn:path_reduction}
  n-Y(\Gamma)=\min\nolimits_{\nu\in\mathcal{R}(\Gamma_+)} d_H(\nu,\Phi_U).
\end{equation}
By Lemma~\ref{lem:switching_burning}, a burning sequence $\mathcal{I}$ of $\Gamma_+$ with final colouring $\nu$ yields the burning sequence $\mathcal{I}^U$ of $\Gamma$ with final colouring $\Phi_U\cdot\nu$. Its green vertices are the vertices $v$ with $\nu(v)\neq\Phi_U(v)$. So Theorem~\ref{thm:path-lower} amounts to exhibiting one reachable colouring of $\Gamma_+$ close to $\Phi_U$, and Theorem~\ref{thm:upper} to bounding the distance to $\Phi_U$ from below for all of them.

On $\Gamma_+$ the influence sum of Remark~\ref{rem:arithmetic} is $h_t(v)=\sum_{u\in N(v)}\nu_{t-1}(u)$, so the spread rule reads: an unburned vertex takes the strict majority colour of its burned neighbours, and stays unburned in case of a tie. On a path this says that an unburned vertex with exactly one burned neighbour takes the colour of that neighbour, and that an unburned vertex with two burned neighbours takes their common colour if they agree, and stays unburned if they do not.

Let $\mathcal{I}\in\mathcal{B}(\Gamma_+)$ and write $\nu=\nu_{\mathcal{I}}$. For a vertex $v$ let $\beta(v)=\min\{t:\nu_{\mathcal{I},t}(v)\neq 0\}$ be the time at the end of which $v$ becomes burned, and put $g(v)=\min_{u\in N(v)}\beta(u)$. The spread step at time $t$ consults exactly the neighbours $u$ of $v$ with $\beta(u)\le t-1$. We say that $v$ is \emph{ignited} if $v$ occurs in $\mathcal{I}$; its ignition time is then $\beta(v)$. 

An ignited vertex $v$ is unburned with two burned neighbours of opposite colours at the end of some time step $t<\beta(v)$ if and only if it has two neighbours
$u,u'$ with $\nu(u)\neq\nu(u')$ and $\beta(u)=\beta(u')<\beta(v)$. Indeed, if the two burn times differed, then $v$ would have exactly one burned neighbour at the
end of the earlier of them and would be burned by spreading at the next time step, giving $\beta(v)\le\beta(u')\le t$. We call such a $v$ \emph{held}, and we call an ignited vertex that is not held \emph{free}. Note that $v$ is free when the tie occurs only at the spread step of time step $\beta(v)$.

\subsection{Proof of Theorem~\ref{thm:path-lower}}\label{sec:lower}
We will construct an ignition sequence witnessing the claimed lower bound.

The map $v_i\mapsto v_{n+1-i}$ is an isomorphism of $\Gamma$ onto the signed path with the reversed signature. It preserves $n$, the orders of the negative segments, the value of $\varepsilon$ and the plurality number. So we may assume throughout this subsection that
if  $S_1$ is a negative segment  with $m_1 \equiv 2  \, (\bmod \, 3)$,
then $\varepsilon=1$. 

Let $S=u_1u_2\cdots u_m$ be a segment, written from left to right, and define a list $L(S)$ of vertices of $S$ by
\[
  L(S)=\begin{cases}
    (u_1), &m=1 \text{ and either } |S'| \equiv 0 \, (\bmod \, 3) \text{ or } S=S_1,\\
    (\,), &m=1 \text{ otherwise},\\
    (u_2), &m=2,\\
    (u_2,u_5,\ldots,u_{m-1}), &m\ge 3 \text{ and } m\equiv 0 \, (\bmod \, 3),\\
    (u_2,u_5,\ldots,u_{m-2},u_m), &m\ge 4 \text{ and } m\equiv 1 \, (\bmod \, 3),\\
    (u_3,u_6,\ldots,u_{m-2},u_m), &m\ge 5 \text{ and } m\equiv 2 \, (\bmod \, 3).
  \end{cases}
\]
The entries of the progression are called the \emph{running entries} of $S$. The entry $u_m$, where present, is called the \emph{closing entry} of $S$. Thus the lists for $m=1$ and $m=2$ consist of a closing entry alone, and a running entry is never the last vertex of its segment. We record four properties, each immediate from the definition.

\smallskip
(K1) The list $L(S)$ has exactly $\lfloor m/3\rfloor$ running entries, and $L(S_1)$ is nonempty.

(K2) Consecutive entries of $L(S)$ are three places apart, except that a closing entry lies two places after its predecessor.

(K3) The first entry of $L(S)$, if there is one, is $u_1$, $u_2$ or $u_3$. It is $u_1$ only when $m=1$, and it is $u_3$ only when $m\ge 5$ and $m\equiv 2\, (\bmod \, 3)$.

(K4) The last entry of $L(S)$ is $u_m$, except when $m\equiv 0\, (\bmod \, 3)$, in which case it is $u_{m-1}$.
\smallskip

Let $L=(w_1,\ldots,w_\ell)$ be the concatenation of $L(S_1),L(S_2),\ldots,L(S_r)$, so that the entries of $L$ increase from left to right. Define a sequence $\mathcal{I}$ of ignitions on $\Gamma_+$ as follows. At time step $k\le\ell$ ignite $w_k$, with the colour $-\Phi_U(w_k)$ if $w_k$ is a running entry of a negative segment, and with the colour $\Phi_U(w_k)$ otherwise. Once $L$ is exhausted, ignite each held vertex $v$ with the colour $\Phi_U(v)$, one per time step. On $\Gamma$ itself this ignites the running entries of the negative segments as green and every other ignited vertex as yellow.  For an entry $w$ of $L$, let $c(w)$ be the colour that $\mathcal{I}$ prescribes for $w$, that is,
\begin{equation}\label{eq:cw}
c(w)=\begin{cases}
-\Phi_U(w), & \text{if $w$ is a running entry of a negative segment},\\
\phantom{-}\Phi_U(w), & \text{otherwise.}
\end{cases}
\end{equation}

We show that $\mathcal{I} \in \mathcal{B}(\Gamma_+)$ whose final colouring $\nu$ differs from $\Phi_U$ exactly at the running entries of the negative segments, and at $v_1$ when $\varepsilon=1$. 
Along the way we show that each $w_k$ is unburned at the ignition step of time step $k$. It is then ignited there, so $\nu(w_k)=c(w_k)$.

\begin{example}\label{ex:path-lower}
Let $n=11$ and $\sigma=(+,+,-,-,-,-,+,+,-,+)$. The segments are
$S_1=v_1v_2$ (positive, order $2$), $S_2=v_3v_4v_5v_6v_7$ (negative, order $5$), $S_3=v_8$ (positive, order $1$), $S_4=v_9v_{10}$ (negative, order $2$) and $S_5=v_{11}$ (positive, order $1$). Here $\varepsilon=0$ and the bound is
$11-\lfloor 5/3\rfloor-\lfloor 2/3\rfloor=10$.
 
The segments $S_1$ and $S_4$ have order $2$, so each contributes its second vertex. The segment $S_2$ has order $5\equiv 2\, (\bmod \, 3)$, so it contributes $u_3=v_5$ and then $u_5=v_7$. The segments $S_3$ and $S_5$ have order $1$ and are preceded by segments of order $\not\equiv 0\, (\bmod \, 3)$, so they contribute nothing. Hence $L=(v_2,v_5,v_7,v_{10})$. We display the process on $\Gamma$ itself, that is, the burning sequence $\mathcal I^U$ of Lemma~\ref{lem:switching_burning}. Its colours are those of $\mathcal I$ multiplied by $\Phi_U$. So a vertex is yellow exactly where $\nu$ agrees with $\Phi_U$. Each line shows one time step: the vertex ignited, and the state at the end of the step. Recall that the spread step comes before the ignition. Writing $\cdot$ for an unburned vertex:
\[
\begin{array}{lllllll}
t=1: & \text{ignite } v_2^{\,y} & \cdot y\cdot\cdot\cdot\cdot\cdot\cdot\cdot\cdot\cdot&\ \ \ &t=4: & \text{ignite } v_{10}^{\,y} & yyy\cdot gyyy\cdot y\cdot\\
t=2: & \text{ignite } v_5^{\,g} & yyy\cdot g\cdot\cdot\cdot\cdot\cdot\cdot
&\ \ \ &t=5: & \text{ignite } v_4^{\,y} & yyyygyyy\cdot yy\\
t=3: & \text{ignite } v_7^{\,y} & yyy\cdot gyy\cdot\cdot\cdot\cdot
&\ \ \ &
t=6: & \text{ignite } v_9^{\,y} & yyyygyyyyyy
\end{array}
\]
The final colouring has ten yellow vertices, meeting the bound. The vertices $v_4$ and $v_9$ are held, at the end of the time steps $2$ and $4$ respectively, and are ignited yellow at no cost in the last phase. 
\end{example}

\begin{lemma}\label{lem:delivery}
Let $w$ be an entry of $L$ and let $z$ be a neighbour of $w$. Then
\[
  c(w)\,\Phi_U(z)=
  \begin{cases}
    -1, &\parbox[t]{4.95in}{if $w$ is the closing entry of a negative segment and $z$ is its left neighbour,}\\
    \phantom{-}1, & \text{otherwise.}
\end{cases}
\]
In particular, if $z$ is burned by spreading with the colour $c(w)$, then
$\nu(z)=-\Phi_U(z)$ in the first case and $\nu(z)=\Phi_U(z)$ in the second.
\end{lemma}
\begin{proof}
Let $S=u_1\cdots u_m$ be the segment containing $w$. Write $c(w)=\theta\,\Phi_U(w)$, so that $\theta=-1$ exactly when $w$ is a running entry of a negative segment and $\theta=1$ otherwise, by \eqref{eq:cw}. Now
$
c(w)\,\Phi_U(z)=\theta\,\Phi_U(w)\Phi_U(z)=\theta\,\sigma(wz)
$. 

Suppose first that $w=u_j$ is a running entry. By the definition of $L(S)$ we have $2\le j\le m-1$, so both neighbours of $w$ lie in $S$ and $wz$ is an edge of $S$. Hence $\sigma(wz)=-1$ exactly when $S$ is negative, which is exactly when $\theta=-1$. In both cases $\theta\,\sigma(wz)=1$. 

Suppose now that $w=u_m$ is the closing entry, so that $\theta=1$ and $c(w)\Phi_U(z)=\sigma(wz)$. If $z$ is the right neighbour of $w$, then $z$ is the first vertex of the segment following $S$. So $wz$ joins two consecutive segments, and $\sigma(wz)=1$. Let $z$ be the left neighbour of $w$. If $m\ge 2$, then $z=u_{m-1}$ and $wz$ is an edge of $S$. So $\sigma(wz)=-1$ when $S$ is negative and $\sigma(wz)=1$ when $S$ is positive. If $m=1$, then $S$ is positive, since negative segments have order at least $2$. Then $z$ lies in the segment preceding $S$, so again $\sigma(wz)=1$. Thus $c(w)\Phi_U(z)=-1$ exactly in the stated case. 

The last assertion follows by using $\Phi_U(z)^2 = 1$.
\end{proof}

\begin{proof}[Proof of Theorem~\ref{thm:path-lower}]
Run $\mathcal{I}$ on $\Gamma_+$, and for $1\le k<\ell$ let $x_k$ be the right neighbour of
$w_k$, which exists because $w_k$ lies to the left of $w_{k+1}$. We prove by induction on
$k\le\ell-1$:

\smallskip
\noindent $(\mathrm{A}_k)$ $\beta(w_k)=k$ and $\beta(x_k)=k+1$. At the end of the spread step of time step $k+1$, every vertex to the left of $x_k$ is burned or held,
and no vertex to its right is burned. Finally, $\nu(z)=\Phi_U(z)$ for every burned vertex $z$ that is not a running entry of a negative segment.

\smallskip
\noindent Two amendments are needed though, both concerning $v_1$ and both affecting $(A_1)$ only. $(\mathrm{a})$ If $m_1\ge5$ and $m_1\equiv2\pmod 3$, that is, if $w_1$ is the third vertex of $S_1$, then $v_1$ is neither burned nor held at the end of the spread step of time step $2$, its burn time being $\beta(v_1)=3$. In this case $L(S_1)$ has a running and a closing entry, so $\ell\ge2$ and $v_1$ is burned by the end of the spread step of time step $\ell+1$. $(\mathrm{b})$ If $S_1$ is negative with $m_1\equiv2\pmod 3$, hence only if $\varepsilon=1$, then $\nu(v_1)=-\Phi_U(v_1)$.

Nothing is burned before the ignition step of time step~$1$. So $w_1$ is ignited there, $\beta(w_1)=1$ and $\nu(w_1)=c(w_1)$. At the spread step of time step $2$ the neighbours of $w_1$ are the only unburned vertices with a burned
neighbour, and each takes the colour $\nu(w_1)$. In particular, the right neighbour $x_1$ of $w_1$, if it exists, burns there with colour $\Phi_U(x_1)$, by Lemma~\ref{lem:delivery}. For $\ell\ge2$ this gives
$\beta(x_1)=2$ and $\nu(x_1)=\Phi_U(x_1)$. By (K3) we have $w_1=u_j$ in $S_1$ with $j\le3$.

Let $j=1$. Then $w_1=v_1$, and no vertex lies to its left. Let $j=2$. Then $\beta(v_1)=2$ and
$\nu(v_1)=\nu(w_1)$, which by Lemma~\ref{lem:delivery} equals $\Phi_U(v_1)$, unless $w_1$ is
the closing entry of a negative segment, that is unless $S_1$ is negative of order $2$. Let
$j=3$. Then $m_1\ge5$ and $m_1\equiv2\,(\bmod\,3)$ by (K3), so $\beta(u_2)=2$ and
$\nu(u_2)=\Phi_U(u_2)$ by Lemma~\ref{lem:delivery}, while $v_1$ has no burned neighbour. Here
$L(S_1)$ has a running entry and a closing entry, so $\ell\ge2$. At the spread step of time
step $3$ the only burned neighbour of $v_1$ is $u_2$, whence $\beta(v_1)=3$ and
$\nu(v_1)=\nu(u_2)=\Phi_U(u_2)$. The edge $v_1u_2$ lies in $S_1$, so this is $\Phi_U(v_1)$ when
$S_1$ is positive and $-\Phi_U(v_1)$ when $S_1$ is negative.

In both cases calling for the two amendments, $S_1$ is negative with $m_1\equiv2\,(\bmod\,3)$, as
announced. Moreover $\nu(w_1)=\Phi_U(w_1)$ unless $w_1$ is a running entry of a negative
segment, by \eqref{eq:cw}. So $(\mathrm{A}_1)$ holds when $\ell\ge2$. When $\ell=1$, the same computation gives the following at the end of the spread step of time step $2$. Every vertex to the left of $w_1$ is burned or held, and so is the right neighbour of $w_1$, if it exists. No other vertex is burned.

Let $1\le k<\ell$ and assume $(\mathrm{A}_k)$. Put $v_p=x_k$, let $w=w_{k+1}$ and let
$S=u_1\cdots u_m$ be the segment containing $w$. We claim that
\begin{equation}\label{eqn:position}
  w=v_{p+\delta} \text{ with } 1\le\delta\le 3, \text{ and } \delta=3 \text{ only if } w=u_3
  \text{ and } v_p \text{ is the vertex preceding } u_1 .
\end{equation}
Suppose first that $w$ is not the first entry of $L(S)$. Then its predecessor $w_k$ lies in
$S$. By (K2) it is three places to the left of $w$, or two when $w$ is the closing entry. As
$v_p$ is the right neighbour of $w_k$, we get $\delta=2$ or $\delta=1$.

Suppose now that $w$ is the first entry of $L(S)$. Then $S\ne S_1$ by (K1), so $S'$ exists. We
claim that $v_p$ is $u_1$ or the vertex preceding $u_1$. Let $L(S')$ be nonempty. Its last entry is $w_k$, which by (K4) is the last vertex of $S'$, or the second to last one when $3$ divides $|S'|$. So $v_p$ is $u_1$ in the first case and the last vertex of $S'$ in the second. Let $L(S')$ be empty. Then $S'$ has order $1$ and is positive, and $S'\neq S_1$ by (K1). So the segment preceding $S'$ exists and is negative, hence of order at least $2$, and its order is not divisible by $3$ by the definition of $L(S')$. By (K4) its list ends at its last vertex, that vertex is $w_k$, and $v_p$ is the vertex of $S'$. This proves the claim. By (K3) we have $w=u_j$ with $j\le3$. Hence $\delta=j-1$ when $v_p=u_1$, and $\delta=j$ when $v_p$ precedes $u_1$. It remains to show $\delta\neq 0$. By the above, $\delta=0$ would require $v_p=u_1$ and $j=1$. Suppose $j=1$. By (K3), $m=1$. Since $S\neq S_1$ and $L(S)$ is nonempty, the definition of $L(S)$ gives $|S'|\equiv 0 \pmod 3$. Thus $|S'|\ge 3$, so $L(S')$ is nonempty and its last entry is $w_k$. By
(K4), $w_k$ is the second to last vertex of $S'$. Hence $v_p$ is the last vertex of $S'$, which precedes $u_1$. So $v_p\neq u_1$, and $\delta\ge 1$. This proves \eqref{eqn:position}.

By $(\mathrm{A}_k)$ no vertex to the right of $v_p$ is burned at the end of the spread step of time step $k+1$. So $w$ is ignited at time step $k+1$, $\beta(w)=k+1$ and $\nu(w)=c(w)$. Moreover $\nu(w)=\Phi_U(w)$ unless $w$ is a running entry of a negative segment, by \eqref{eq:cw}. Consider the spread step of time step $k+2$. Three kinds of unburned vertex have a burned neighbour there. The held ones stay unburned. The $\delta-1\le2$ vertices between $v_p$ and
$w$ are treated below. The right neighbour $z$ of $w$, if it exists, has $w$ as its only burned neighbour, so it is burned at time step $k+2$ with $\nu(z)=\Phi_U(z)$, by
Lemma~\ref{lem:delivery}.

Let $\delta=1$. Then no vertex lies between $v_p$ and $w$. Let $\delta=3$. Then $w=u_3$ and
$v_p$ precedes $u_1$ by \eqref{eqn:position}. So the edge $v_pu_1$ joins two consecutive
segments, and $\Phi_U(u_1)=\Phi_U(v_p)$. At the start of this spread step $u_1$ has $v_p$ as
its only burned neighbour, and $u_2$ has $w$ as its only burned neighbour. Hence
$\nu(u_1)=\nu(v_p)=\Phi_U(u_1)$, and $\nu(u_2)=\nu(w)=\Phi_U(u_2)$ by
Lemma~\ref{lem:delivery}, $w$ being a running entry.

Let finally $\delta=2$. Then $v_{p+1}$ lies between the burned vertices $v_p$ and $w$, which
deliver the colours $\Phi_U(v_p)$ and $\nu(w)$. If the two deliveries are opposite, then
$v_{p+1}$ is held and there is nothing to prove. If they agree, then $v_{p+1}$ is burned with
the colour $\nu(w)$, and we must check that $\nu(w)=\Phi_U(v_{p+1})$. By
Lemma~\ref{lem:delivery} this holds unless $w$ is the closing entry of a negative segment $S$
and $v_{p+1}$ is its left neighbour. Assume that it is. Then $w$ is the first entry of $L(S)$. Otherwise, by (K2), the closing entry $w$ lies two places after its predecessor, which gives $\delta=1$. So $L(S)$ has no running entry, whence $m\le2$ by (K1), and $m=2$ because $S$ is negative. Thus $w=u_2$ and $v_{p+1}=u_1$. Also $v_p$ precedes $u_1$, because $\delta=2$, so the edge $v_pu_1$ joins two consecutive segments and $\Phi_U(v_p)=\Phi_U(u_1)$. On the other
hand $\nu(w)=\Phi_U(u_2)$ by \eqref{eq:cw}, and $\Phi_U(u_2)=-\Phi_U(u_1)$ because
the edge $u_1u_2$ lies in the negative segment $S$. The two deliveries are therefore opposite,
and this case does not arise.

This proves $(\mathrm{A}_{k+1})$ when $k+1\le\ell-1$. For $k=\ell-1$, the argument gives the following at the end of the spread step of time step $\ell+1$. Every vertex to the left of $w_\ell$ is burned or held, and so is the right neighbour of $w_\ell$, if it exists. No other vertex is burned. By the base case, the same holds when $\ell=1$.

Now $w_\ell\in\{v_{n-1},v_n\}$. Indeed, let $R$ be the segment containing $w_\ell$. Every
segment to the right of $R$ has an empty list, hence order $1$, and is therefore positive. Two consecutive segments are never both positive, so at most one segment lies to the right of $R$. If there is none, then $R=S_r$ and, by (K4), $w_\ell$ is the last or the second to last vertex of $R$. If there is one, say $R''$, then $R''=S_r$ has order $1$, so $R''=\{v_n\}$; moreover $L(R'')$ is empty with $R''\neq S_1$, so $3$ does not divide the order of $R$, and by
(K4) the vertex $w_\ell$ is the last vertex of $R$, that is $v_{n-1}$.

So every vertex is burned or held once the spread step of time step $\ell+1$ is over. The remaining ignitions therefore take only held vertices, one per time
step, and each such vertex $v$ is ignited with colour $\Phi_U(v)$. None of them causes a vertex to burn by spreading. Hence $\mathcal{I}\in\mathcal{B}(\Gamma_+)$.

By the above and \eqref{eq:cw}, the colouring $\nu$ differs from $\Phi_U$ exactly at the running entries of the negative segments and, in the amended case $\mathrm{(b)}$, at $v_1$. That case occurs exactly when $\varepsilon=1$. No running entry is the first vertex of its segment, so
$v_1$ is not one of them. By (K1) a negative segment of order $m$ has $\lfloor m/3\rfloor$
running entries, so
\[
d_H(\nu,\Phi_U)=\sum_{i=1}^{s}\Big\lfloor\frac{m_i}{3}\Big\rfloor+\varepsilon ,
\]
and \eqref{eqn:path_reduction} gives the stated bound.
\end{proof}

\subsection{Proof of Theorem~\ref{thm:upper}}\label{sec:upper}
Let $\mathcal{I}\in\mathcal{B}(\Gamma_+)$ be arbitrary and write $\nu=\nu_{\mathcal{I}}$. The \emph{blocks} of $\nu$ are the maximal subpaths of $P_n$ on which $\nu$ is constant; consecutive blocks have opposite colours, and a
block of order $1$ is a \emph{singleton block}. For a segment $S$, the \emph{pieces} of $\nu$ on $S$ are the nonempty intersections of the blocks with $V(S)$. They are consecutive subpaths of $S$, and consecutive pieces have opposite colours.

\begin{lemma}\label{lem:wave}
Let $v$ be a vertex of $P_n$.
\begin{enumerate}
\item[\textup{(W1)}] If $v$ is not ignited, then $\beta(v)=g(v)+1$, and every neighbour $u$ of $v$ with $\beta(u)=g(v)$ has $\nu(u)=\nu(v)$.
\item[\textup{(W2)}] If $v$ is free, then $\beta(v)\le g(v)$.
\item[\textup{(W3)}] Distinct ignited vertices have distinct values of $\beta$.
\end{enumerate}
\end{lemma}

\begin{proof}
Put $g=g(v)$ and suppose that $\beta(v)\geq g+1$. Then $v$ is unburned before the spread step of time $g+1$, and its burned neighbours at that moment are the neighbours $u$ with $\beta(u)=g$. There is at least one such neighbour. If all of them have the same colour $c$, then $v$ burns by spreading at time $g+1$ with colour $c$. In that case $v$ is not ignited, because the ignition step only chooses unburned vertices. Otherwise $v$ has two neighbours $u,u'$ with $\beta(u)=\beta(u')=g$ and $\nu(u)\neq\nu(u')$. These are all the neighbours of $v$, and they remain burned, so their influences cancel at time $g+1$ and afterwards; and $v$ never burns by spreading. Hence $v$ is ignited, and it is held.

(W1) A vertex burned by spreading at time $t$ has a burned neighbour at the start of time $t$, so $\beta(v)\geq g+1$ when $v$ is not ignited. The first alternative above then occurs, and it gives both assertions.

(W2) If $v$ were ignited with $\beta(v)\geq g+1$, the first alternative would be impossible, so $v$ would be held. Hence a free vertex satisfies $\beta(v)\le g$.

(W3) At most one vertex is ignited per time step.
\end{proof}

\begin{lemma}\label{lem:short}
\begin{enumerate}[{\upshape(i)}]
\item If $\{x\}$ is a singleton block, then $x$ is free, and $\beta(x)\le\beta(u)$ for every neighbour $u$ of $x$.
\item Let $B$ be a block of order $3$ with middle $w$, put $M=\beta(w)$, let $u$ be an end of
$B$, and let $u'$ be the neighbour of $u$ outside $B$, assumed to exist. Then
$u$ is held if and only if $u$ is ignited with $\beta(u)\geq M+1$, and in that case $\beta(u')=M$.
If $u$ is not held, then $\beta(u)\le\beta(u')$; moreover $\beta(u)=M+1$ if $u$ is not ignited, and $\beta(u)\le M$ if $u$ is free. In particular, $\beta(u)\geq M+1$ if and only if $u$ is not free.
\item With $B$, $w$ and $M$ as in (ii), $w$ is never held, and $w$ is ignited if and only if $\beta(u)\ge M$
for both ends $u$ of $B$.
\item Let $u$ and $u'$ be adjacent vertices in different blocks. Suppose that each of $u$ and $u'$ is a singleton block or an end of a block of order 3. If neither $u$ nor $u'$ is held, then $\beta(u)=\beta(u')$. In particular, no two singleton blocks are adjacent.
\end{enumerate}
\end{lemma}

\begin{proof}
\begin{enumerate}[{\upshape(i)}]
    \item Recall from (W1) that a vertex which is not ignited shares its colour with a neighbour. Since $n\ge 2$, the vertex $x$ has a neighbour, and every neighbour of $x$ has colour $-\nu(x)$. Hence $x$ is ignited. As its neighbours have equal colours, $x$ is not held. Now (W2) gives $\beta(x)\le g(x)$. 
    \item Let $d$ be the colour of $B$. The neighbours of $u$ are $w$ and $u'$, with $\nu(w)=d$ and $\nu(u')=-d$, since $u'$ lies in the block next to $B$. So $u$ is held exactly when $u$ is ignited and $\beta(u')=M<\beta(u)$. In particular, a held vertex $u$ is ignited with $\beta(u)\geq M+1$ and $\beta(u')=M$. Conversely, if $u$ is ignited with $\beta(u)\geq M+1$, then $g(u)\le M<\beta(u)$, so $u$ is not free by (W2), that is, $u$ is held. Now let $u$ be not held. If $u$ is not ignited, (W1) gives $\beta(u)=g(u)+1$, and every neighbour $z$ of $u$ with $\beta(z)=g(u)$ has colour $d$. So $\beta(u')\neq g(u)$. Hence $g(u)=M<\beta(u')$ and $\beta(u)=M+1\le\beta(u')$. If $u$ is free, (W2) gives $\beta(u)\le g(u)=\min\{M,\beta(u')\}$. The last assertion follows, because a vertex that is not free is either held or not ignited.
    \item Both neighbours of $w$ have colour $d$, so $w$ has no two neighbours of opposite
    colours and is never held. If $w$ is ignited, then $w$ is free, and (W2) gives $M\le\beta(u)$ for both ends $u$. If $w$ is not ignited, (W1) gives $M=g(w)+1$, so some end $u$ has $\beta(u)=M-1<M$.
    \item Note that $u'$ is the neighbour of $u$ outside the block of $u$, and vice versa. So (i) or (ii), applied to $u$, gives $\beta(u)\le\beta(u')$. Applied to $u'$, it gives $\beta(u')\le\beta(u)$. For the last assertion, let $u$ and $u'$ be adjacent singleton blocks. Both are free by (i), so they are distinct ignited vertices with $\beta(u)=\beta(u')$, contrary to (W3). \qedhere
\end{enumerate}
\end{proof}

\begin{lemma}\label{lem:chain}
Let $\{x\}$ and $\{y\}$ be singleton blocks, with $x$ to the left of $y$ and no singleton block between them. Then $x$ and $y$ are not adjacent, and some block between $x$ and $y$ has order $2$ or at least $4$.
\end{lemma}
\begin{proof}
By Lemma~\ref{lem:short}(iv), $x$ and $y$ are not adjacent. So the vertices strictly between them form blocks $B_1,\ldots,B_k$, listed from left to right, where $k\ge1$ and each $B_i$ has order more than 1. Suppose, for a contradiction, that all of them have order $3$, and write $B_i=\{a_i,b_i,c_i\}$ from left to right. Put $c_0=x$, $a_{k+1}=y$ and $M_i=\beta(b_i)$ for $1\le i\le k$. For $0\le i\le k$ the vertices $c_i$ and $a_{i+1}$ are adjacent and lie in different blocks, and each of them is a singleton block or an end of a block of order $3$. So Lemma~\ref{lem:short}~(ii) and (iv) apply to the pair $c_i,a_{i+1}$.

Let $X_i=\{c_0\}\cup B_1\cup\cdots\cup B_i$, and let $T_i$ be the set of values $\beta(z)$ over the ignited vertices $z\in X_i$. By (W3),
\begin{equation}\label{eq:chain_distinct}
  \text{if } z\notin X_i \text{ is ignited, then } \beta(z)\notin T_i .
\end{equation}
Moreover, for $1\le i\le k$,
\begin{equation}\label{eq:chain_middle}
  \text{if neither } a_i \text{ nor } c_i \text{ is free, then } b_i \text{ is ignited, and hence } M_i\in T_i .
\end{equation}
Indeed, the last assertion of Lemma~\ref{lem:short}~(ii) gives $\beta(a_i)\geq M_i+1$ and $\beta(c_i)\geq M_i+1$, and Lemma~\ref{lem:short}~(iii) then shows that $b_i$ is ignited.

We prove by induction on $i\in\{0,1,\ldots,k\}$ the two statements
\[
  (\mathrm{A}_i)\ \ a_i \text{ is not free, when } i\ge 1;
  \qquad\qquad
  (\mathrm{B}_i)\ \ \beta(c_i)\in T_i .
\]
Statement $(\mathrm{A}_0)$ is empty, and $(\mathrm{B}_0)$ holds because $c_0=x$ is ignited by Lemma~\ref{lem:short}(i).

Assume $(\mathrm{A}_i)$ and $(\mathrm{B}_i)$ for some $0\le i\le k$. We first show
\begin{equation}\label{eq:chain_next}
  \text{if } a_{i+1} \text{ is not held, then } \beta(a_{i+1})\in T_i .
\end{equation}
If $c_i$ is not held, Lemma~\ref{lem:short}(iv) gives $\beta(a_{i+1})=\beta(c_i)$, which lies in $T_i$ by $(\mathrm{B}_i)$. If $c_i$ is held, then $i\ge1$, since $c_0$ is free. Lemma~\ref{lem:short}(ii) gives $\beta(a_{i+1})=M_i$. As $c_i$ is held and $a_i$ is not free by $(\mathrm{A}_i)$, \eqref{eq:chain_middle} gives $M_i\in T_i$.

Next, $a_{i+1}$ is not free. Otherwise $a_{i+1}$ is ignited and not held, so $\beta(a_{i+1})\in T_i$ by \eqref{eq:chain_next}. Since $a_{i+1}\notin X_i$, this contradicts \eqref{eq:chain_distinct}. When $i<k$ this is $(\mathrm{A}_{i+1})$.

Let $i<k$. It remains to prove $(\mathrm{B}_{i+1})$. If $c_{i+1}$ is ignited, then $\beta(c_{i+1})\in T_{i+1}$ by definition. So assume $c_{i+1}$ is not ignited. Then $c_{i+1}$ is not free, and \eqref{eq:chain_middle} together with $(\mathrm{A}_{i+1})$ shows that $b_{i+1}$ is ignited, with $\beta(b_{i+1})=M_{i+1}$. If $a_{i+1}$ were held, Lemma~\ref{lem:short}(ii) would give $\beta(c_i)=M_{i+1}=\beta(b_{i+1})$. This is impossible by \eqref{eq:chain_distinct}, since $\beta(c_i)\in T_i$ and $b_{i+1}\notin X_i$. So $a_{i+1}$ is not held, and since it is not free, it is not ignited. Lemma~\ref{lem:short}(ii), applied to $a_{i+1}$ and to $c_{i+1}$, gives $\beta(a_{i+1})=M_{i+1}+1=\beta(c_{i+1})$. By \eqref{eq:chain_next}, $\beta(c_{i+1})=\beta(a_{i+1})\in T_i\subseteq T_{i+1}$. This completes the induction.

Now $(\mathrm{A}_k)$ and $(\mathrm{B}_k)$ hold, so the argument above with $i=k$ shows that $a_{k+1}=y$ is not free. This contradicts Lemma~\ref{lem:short}(i).
\end{proof}

We are now in a position to prove Lemma~\ref{lem:odd}, the final missing piece to the proof of Theorem~\ref{thm:upper}.

\begin{lemma}\label{lem:odd}
Let $S$ be a negative segment of order $m$, put $e=e(S)$, and let $q$ be the number of pieces of $\nu$ on $S$ of odd order. Then $3q\le m+6-2e$.
\end{lemma}
\begin{proof}
Write $S=u_1u_2\cdots u_m$ from left to right, and let $L_1,\ldots,L_p$ be the orders of the pieces of $\nu$ on $S$. For a positive integer $L$ put $w(L)=L-3$ if $L$ is odd and $w(L)=L$ if $L$ is even. Since $\sum_j L_j=m$, we have $\sum_j w(L_j)=m-3q$, so it suffices to show that $\sum_j w(L_j)\ge 2e-6$. Note that $w(1)=-2$, $w(3)=0$, and $w(L)\ge 2$ for every other $L$.

Let $s_1$ be the number of pieces of order $1$
and let $t$ be the number of singleton blocks contained in $S$. Every singleton
block contained in $S$ is a piece of order $1$. Conversely, a piece of order $1$ at $u_j$ with $2\le j\le m-1$ is a singleton block, since the pieces on both sides of it have the opposite colour. The same holds at $u_1$ if $u_1=v_1$, and at $u_m$ if $u_m=v_n$, since then that vertex has only one neighbour, and this neighbour lies in $S$. The positions $u_1$ and $u_m$ are
distinct, and exactly $2-e$ of them fail to be endpoints of $P_n$. Hence
$s_1\le t+2-e$.

Let $\{x_1\},\ldots,\{x_t\}$ be the singleton blocks contained in $S$, from left to right. A singleton block between $x_i$ and $x_{i+1}$ would lie in $S$, so there is none. By Lemma~\ref{lem:chain}, for each $i<t$ some block $C_i$ between $x_i$ and $x_{i+1}$ has order $2$ or at least $4$. This block lies in $S$, so it is a piece, and $w(|C_i|)\ge 2$. The pieces $C_1,\ldots,C_{t-1}$ are distinct, and none of them has order $1$. Every other piece has weight $-2$ if it has order $1$, and weight at least $0$ otherwise. Therefore
\begin{align*}
  \sum_j w(L_j)&\;\ge\;-2s_1+2\max\{t-1,0\}\;\ge\;-2s_1+2(t-1)
  \;\ge\;-2s_1+2(s_1-2+e)-2\\ &\;=\;2e-6. \qedhere
\end{align*} 
\end{proof}

With Lemma~\ref{lem:odd}, we can prove Theorem~\ref{thm:upper}.

\begin{proof}[Proof of Theorem~\ref{thm:upper}]
By \eqref{eqn:path_reduction} it suffices to show that $d_H(\nu,\Phi_U)\ge\sum_{i}\lfloor (m_i-1+e_i)/3\rfloor$ for the final colouring $\nu$ of every $\mathcal{I}\in\mathcal{B}(\Gamma_+)$. The negative segments are pairwise vertex disjoint. So it is enough to consider a single negative segment $S$ of order $m$, with $e=e(S)$. Let $D_S$ be the number of vertices $v\in V(S)$ with $\nu(v)\neq\Phi_U(v)$. We
show that $D_S\ge\lfloor(m-1+e)/3\rfloor$.

Let $L_1,\ldots,L_p$ be the orders of the pieces of $\nu$ on $S$, and let $q$ be the number of odd $L_j$. On a piece, $\nu$ is constant while $\Phi_U$
alternates. So $\nu$ and $\Phi_U$ disagree on at least $\lfloor L_j/2\rfloor$ vertices of the $j$th piece. Since $\sum_j L_j=m$, and $L_j-2\lfloor L_j/2\rfloor$ equals $1$ if $L_j$ is odd and $0$ otherwise, Lemma~\ref{lem:odd} gives
\[
  D_S\;\ge\;\sum_{j=1}^{p}\Big\lfloor\frac{L_j}{2}\Big\rfloor\;=\;\frac{m-q}{2}\;\ge\;\frac12\Big(m-\frac{m+6-2e}{3}\Big)\;=\;\frac{m+e-3}{3}.
\]
As $D_S$ is an integer,
$D_S\ge\lceil (m+e-3)/3\rceil=\lfloor (m+e-1)/3\rfloor$, as required.
\end{proof}

\section{Computational Complexity}\label{sec:complexity}

In this section, we study the hardness of various graph burning problems on signed graphs. By varying the constraints placed on the graph structure and the burning process, we obtain a family of four related problems. The underlying graph could be directed or undirected, and the deadline by which the burning process is terminated could be unbounded (that is, the process continues until the entire graph is burned) or bounded below $n$, the number of vertices. 

We first highlight the changes needed in the burning process (see Definition \ref{defn:signed_burning}) to accommodate the directed graphs and deadlines. Let $\Gamma = (G, \sigma)$ be a signed graph, and let $k$ denote the deadline. If there are no deadlines, $k$ is simply set to $n$. In the case of a directed $G$, the fire can only spread in the direction of the arcs of $G$. To reflect that, we modify the spread rule by replacing $N^{+}(v)$ with $N_{\textsf{in}}^{+}(v)$ and $ N^{-}(v)$ with $ N_{\textsf{in}}^{-}(v)$ when calculating $s_t^{y}(v)$ and $s_t^{g}(v)$, the yellow and green influence of an unburned vertex $v$. Here, $N_{\textsf{in}}(v)$ denotes the set of vertices with a directed arc towards $v$. This process terminates either after the ignition step for time $t = k$ or when every vertex is burned yellow or green.

This section focuses on the plurality parameter, which we generalize here.

\begin{definition}[Generalized Plurality]
Let $\Gamma = (G,\sigma)$ be a signed graph, and let $k$ be the deadline for the burning process. Then the plurality number is
    \begin{align*}
        Y_k(\Gamma) \;=\; \max\nolimits_{\mathcal{I}\,\in\,\mathcal{B}_{\leq k}(\Gamma)}\left(\max\bigl\{|\mathcal{Y}_{\mathcal{I}}|,\, |\mathcal{G}_{\mathcal{I}}|\bigr\}\right),
    \end{align*}
where $\mathcal{B}_{\leq k}(\Gamma)$ denotes the set of all burning sequences of $\Gamma$ of length at most $k$.
\end{definition}
In the settings without deadlines, the definition of plurality remains unchanged (see Definition \ref{defn:parameters}). We study the problem of maximizing $Y(\Gamma)$ in settings without a deadline (for both directed and undirected graphs $G$), and maximizing $Y_k(\Gamma)$ for the settings with a deadline (again, for both directed and undirected $G$). These problems form the hierarchy depicted in Figure \ref{fig:burning-hierarchy}, ordered by Cook-reducibility. Our main objective from studying the complexity of the hierarchy is to show the hardness of computing $Y(\Gamma$) exactly in polynomial time for any given undirected signed graph $\Gamma$. 

\begin{figure}[t]
  \centering
  \resizebox{0.7\textwidth}{!}{
  \begin{tikzpicture}[
      node distance = 7mm and 11mm,
      box/.style = {draw, rounded corners, align=center,
                    inner sep=6pt, minimum height=10mm, font=\small},
      redge/.style = {-{Stealth[length=2mm]}, thick},
    ]
    \node[box] (top)   {Directed Signed Burning\\with Deadlines};
    \node[box] (left)  [below left=of top]  {Directed Signed Burning\\without Deadlines};
    \node[box] (right) [below right=of top] {Undirected Signed Burning\\with Deadlines};
    \node[box] (bot)   [below right=of left] {Undirected Signed Burning\\without Deadlines};

    \draw[redge] (left)  -- (top);
    \draw[redge] (right) -- (top);
    \draw[redge] (bot)   -- (left);
    \draw[redge] (bot)   -- (right);

  \end{tikzpicture}}
  \caption{Hierarchy of plurality optimization problems on signed graph burning problems.
  An arrow $A \to B$ indicates that $A$ Cook-reduces to $B$, so hardness transfers upward.}
  \label{fig:burning-hierarchy}
\end{figure}
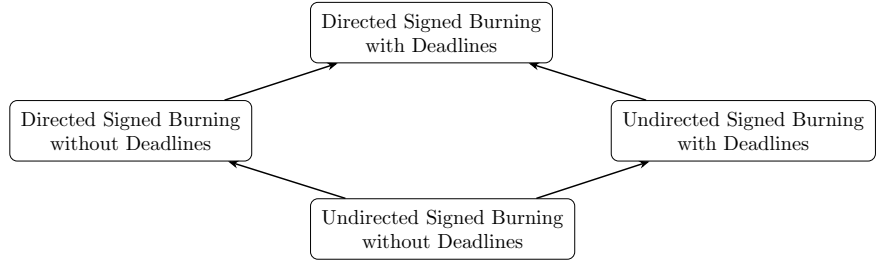

First, we establish the hardness for the top of this hierarchy, computing $Y_k(\Gamma)$ for any given directed signed graph $\Gamma = (G, \sigma)$ and a deadline of $k$. In fact, we prove a stronger result by showing the hardness of approximation for the unsigned graph $G$, or equivalently, when all the signs in $\sigma$ are positive.

\begin{theorem}\label{thm:DSGBwD_hardness}
Unless $\textsf{P} = \textsf{NP}$, there is no polynomial-time $(1 - \frac{1}{e} + \varepsilon)$-approximation algorithm for computing $Y_k(\Gamma)$ for any $\varepsilon > 0$, given a directed signed graph $\Gamma$ and integer $k$ as input.
\end{theorem}

\begin{proof}
We reduce from the Maximum Coverage problem. An instance consists of a set system $(U, \cS)$ and an integer $k$, where $U = \{e_1, e_2, \ldots, e_n \}$ is the ground set and $\cS = \{S_1, S_2, \ldots, S_m \}$ is a family of non-empty subsets of $U$. The task is to select a subfamily $\cS^* \subseteq \cS$ with $\abs{S^*} = k$ maximizing $\abs{\bigcup_{S \in \mathcal{S^*}} S}$. Let $\textsf{OPT}_\textsf{MC}$ denote the maximum. We may assume that $k < m$, otherwise $\cS^* = \cS$. Given such an instance, we construct the directed signed graph $\Gamma = (D, \sigma)$. Label the sets in $\cS$ with indices from $[m]$ and let $\ell = (k + 1)^2$. Consider a bipartite directed graph $D = (V, A)$ with $V = V_\cS \sqcup V_{U}$, where $V_{\cS} = \{a_1, a_2, \ldots, a_m\}$ has one vertex per set $S_i \in \cS$ and $V_U = V_U^1 \sqcup V_U^2 \sqcup \ldots \sqcup V_U^\ell$, where for each $t \in [\ell]$, $V_{U}^t = \{b_1^t, b_2^t, \ldots, b_n^t\}$ contains one vertex per element $e_j \in U$. For every $t \in [\ell]$, we add a directed arc $(a_i, b^t_j)$ to $A$ whenever $e_j \in S_i$ and set $\sigma_{(a_i, b^t_j)} = +1$. Finally, set the deadline to $k + 1$. First, we show the following.
\begin{claim} 
$Y_{k + 1}(\Gamma) = \ell \cdot \textsf{OPT}_{\textsf{MC}} + (k + 1) = \ell \left(\textsf{OPT}_{\textsf{MC}} + \frac{1}{k + 1}\right)$ 
\end{claim}
    
\begin{proof}[Proof of Claim]
\renewcommand{\qedsymbol}{$\Box$\ \emph{(Claim.)}}

We distinguish between two cases:
    
\textbf{Case 1:} Assume that $\textsf{OPT}_{\textsf{MC}} < n$. Clearly $Y_{k + 1}(\Gamma) \geq  \ell \cdot \textsf{OPT}_{\textsf{MC}} + (k + 1)$ since we can ignite $k$ vertices from $V_{\cS}$ corresponding to the optimum of the Maximum Coverage problem for the first $k$ rounds and burn some arbitrary vertex in the last round (note that the fire ignited in round $T = k + 1$ does not spread), resulting in exactly  $\ell \cdot \textsf{OPT}_{\textsf{MC}} + (k + 1)$  yellow vertices at the end of round $k + 1$. Now assume that $Y_{k + 1}(\Gamma) >  \ell \cdot \textsf{OPT}_{\textsf{MC}} + (k + 1)$. If all of the vertices burned by the optimum in the first $k$ rounds are from the set $V_{\cS}$, then by selecting the corresponding sets from $\cS$, we cover more than $\textsf{OPT}_{\textsf{MC}}$ elements from $U$, which is a contradiction (this is possible because we assumed $\textsf{OPT}_{\textsf{MC}} < n$). Consider a scenario in which an unburned $b_j^t$ is ignited. If the vertex has an in-degree of at least $1$, then burning any $a_i$ with an arc $(a_i, b^t_j)$ strictly increases the number of vertices burned yellow (if $b^t_j$ is unburned, no other $b^{t'}_j$ could be burned in $D$ either, and $\ell > 1$). If the in-degree of $b_j^t$ is $0$, burning any arbitrary $a_i$ strictly improves the number of yellows since we assumed $S \neq \varnothing$ for all $S \in \cS$, so $a_i$ has an out-degree of at least $\ell$ (note that this is possible since we assumed $k < m$). In any case, we get a contradiction.

\textbf{Case 2:} Next, let $\textsf{OPT}_{\textsf{MC}} = n$, and for the sake of contradiction, assume that $\textsf{OPT}_{\textsf{SGM}} > \ell \cdot n + k + 1$. Note that igniting the vertices corresponding to the sets in the optimum solution of the Maximum Coverage problem in the first $k$ rounds and igniting any arbitrary vertex from the set $V_{\cS}$ in round $k + 1$ will burn exactly the $\ell \cdot n + k + 1$ vertices yellow (i.e., the entire $V^t_U$ sets for the $t \in [\ell]$ and $k + 1$ vertices of $V_{\cS}$). If the optimum burns the entirety of $V_{U}^1, V_{U}^2, \ldots, V_{U}^\ell$, then the maximum number of vertices that could burn yellow is clearly $\ell \cdot n + k + 1$. This implies that in the first $k$ rounds, the optimum ignites a set of $k' \leq k$ vertices of $V_{\cS}$ that are distinct from the $k$ vertices corresponding to any optimum of the Maximum Coverage problem. Assume that these $k'$ sets cover at most $n' \le n - 1$ elements from $U$ in the Maximum Coverage problem. Therefore, the number of vertices burned yellow by the corresponding $V_{\cS}$ vertices is at most $k' + \ell \cdot n'$. By doing this, the optimum can ignite $k + 1 - k' \leq k + 1$ many vertices from outside $V_{\cS}$, making the total number of yellow vertices $k' + \ell \cdot n' + k + 1 \leq 2k + 1 + \ell \cdot (n - 1) = 2k + 1 + \ell \cdot n -  (k + 1)^2 < \ell \cdot n + k + 1$ for $k >= 1$. This is a contradiction to its optimality.
\end{proof}
    
Feige \cite{Feige98} showed that for any $\varepsilon > 0$, the Maximum Coverage problem cannot be approximated in polynomial time within a ratio of $(1 - 1/e + \varepsilon)$ unless $\textsf{P} = \textsf{NP}$. To show the hardness of approximation, we look further into Feige's hardness proof, where he provides a gap-producing reduction from the 3\textsf{SAT}-5\footnote{An instance of the the 3\textsf{SAT}-5 problem consists of a set of literals $x_1, x_2, \ldots, x_n$ and a set of $3$-\textsf{CNF} clauses $C_1, C_2, \ldots, C_m$ with the promise that no literal appears, positively or negatively, more than $5$ times in the clauses. The task is to decide whether a satisfying assignment of boolean values to literals exists.} problem to Maximum Coverage as follows: If the instance of the 3\textsf{SAT}-5 problem is satisfiable, then the instance of the Maximum Coverage problem has a solution that covers the entire set $U$. If the 3\textsf{SAT}-5 instance is not satisfiable, then any choice of $k$ sets from $\cS$ would cover at most a $1 - 1/e + \littleO{1}$ fraction of the elements of $U$. Following the chain of reduction, if the 3\textsf{SAT}-5 instance is satisfiable, we get an instance of the Directed Signed Graph Burning problem where $Y_{k + 1}(\Gamma) = \ell \cdot n + k + 1 = \ell \cdot \left(n + \frac{1}{k + 1}\right) = \ell \cdot (n + \littleO{1})$. Otherwise, $Y_{k + 1}(\Gamma) \leq \ell \left(\textsf{OPT}_{\textsf{MC}} + \frac{1}{k + 1}\right) = \ell \left((1 - 1/e + \littleO{1}) \cdot n + \frac{1}{k + 1}\right) = \ell (1 - 1/e + \littleO{1}) \cdot n$ (note that $\frac{1}{k + 1} = \littleO{1}$ since $k$ tends to infinity). Therefore, any algorithm that approximates $Y_{k + 1}(\Gamma)$ within a factor of $(1 - 1/e + \varepsilon)$ for any $\varepsilon > 0$ could decide whether the 3\textsf{SAT}-5 instance is satisfiable or not, which is not possible unless $\textsf{P} = \textsf{NP}$.
\end{proof}

\begin{theorem}\label{thm:USGBwD_hardness}
    Unless $\textsf{P} = \textsf{NP}$, there is no polynomial-time exact algorithm for computing $Y_k(\Gamma)$, given a undirected signed graph $\Gamma$ and integer $k$ as input.   
\end{theorem}
\begin{proof}
The result follows in a straightforward manner from the hardness  of the standard (unsigned) graph burning. In the decision version of the standard unsigned graph burning, we are given an undirected graph $G$ on $n$ vertices and an integer $k$. The task is to decide if a sequence of at most $k$ ignitions exist that can burn the entire vertex set by round $k$, which is an \textsf{NP}-complete problem \cite{BessyBJRR17}.

One can reduce any arbitrary instance of the standard unsigned graph burning problem given as graph $G$ to an instance of computing $Y_k(\Gamma)$ on an undirected signed graph $\Gamma = (G, \sigma)$ with a deadline by letting $\sigma(e) = +1$ for all $e \in E(G)$ and setting the deadline to $k$. Clearly, $Y_k(\Gamma) = n$ if and only if an ignition sequence of size at most $k$ in the original unsigned graph can burn the entire vertex set.
\end{proof}

\begin{theorem}\label{thm:DSGBwoD_hardness}
    Unless $\textsf{P} = \textsf{NP}$, there is no polynomial-time exact algorithm for computing $Y(\Gamma)$, given a directed signed graph $\Gamma$ as input.   
\end{theorem}
\begin{proof}
We reduce from the Set Cover problem. In an instance of the Set Cover problem, we are given a set $U$ of $n$ elements, a family $\cS$ of $m$ subsets of $U$, and a parameter $k$. The task is to decide whether there exists a set $\cS'$ of exactly\footnote{The standard Set Cover definition only requires $\cS'$ to be of size at most $k$, but it is well understood that it can be reduced to the case with cover size of exactly $k$ by adding new elements and singleton sets containing those elements.} $k$ sets in $\cS$ that collectively cover $U$, i.e., $\bigcup_{s \in \cS'} s = U$. Such an $\cS'$ is said to ``cover'' $U$. Without loss of generality, we can assume that no set in $\cS$ contains another set. Furthermore, we can assume that every set in $\cS$ has a cardinality of at least $2$ since all the empty sets could be removed and singleton sets must be added to the solution. We can also assume that $U$ has no isolated (not in any $S \in \cS$) elements.

We generate a directed signed graph $\Gamma$ with three main parts. The first is a bipartite digraph $H = (V, E)$ representing the set system with $V = V_U \sqcup V_{\cS}$, where $V_U$ contains one vertex per element in $U$ and $V_{\cS}$ contains one vertex per set in $\cS$. The edges are directed from $V_{\cS}$ to $V_U$. There is a positive edge from $s \in V_\cS$ to $u \in V_U$ if and only if $u \in s$. We let $v_1, v_2, \ldots, v_N$ collectively denote the nodes in $H$ with $N = m + n$. We add two gadgets to $H$ in this reduction to complete $\Gamma$. These gadgets are depicted in Figure \ref{fig:timer}. 
    
\emph{i) The timer gadget $T$:} It is depicted on the left side in Figure \ref{fig:timer}, consisting of a root node $r$ and directed paths $p_1, p_2, \ldots, p_N$, each of length $k - 1$ (i.e., $k$ vertices in each). For each $i \in [N]$, let $s_i$ and $t_i$ denote the first and last vertices of $p_i$, respectively. $r$ is connected to each $s_i$ via a positive edge and every edge of the path is also a positive edge. Each $t_i$ is connected to $v_i$ via a negative edge. Let $v_u \in V_U$ correspond to an element $u \in U$. We say a vertex $v_{u}$ is \emph{marked for coverage} in a given round if, in the previous round, a vertex $v_S \in V_{\cS}$ was ignited yellow, where $S \in \cS$ contains $u$. The function of the timer gadget is to ensure that every vertex of $V_U$ that has not been marked for coverage by round $k + 2$ turns green at the end of it. As a result, the timer gadget simulates a deadline for any burning sequence to burn vertices yellow by round $k + 2$. Note that the vertices in $V_U$ t hat are marked for coverage at time $k + 2$ remain unburned because of a tie in the influence they receive; a yellow influence from the $V_S$ that has marked them for coverage and a green influence from a $t_i$. The optimal burning sequence will then proceed to burn them yellow one by one, so they will eventually contribute to $Y(\Gamma)$.
    
\emph{ii) The bait gadget $B$:} It is shown on the right side in Figure \ref{fig:timer}, consisting of a large number $M$ of vertices (consider $M \geq N^3$). Every vertex $w \in (V(T) - \{r\}) \cup V(H) \cup V(B)$ is connected to $M/2$ arbitrarily chosen vertices of $B$ via positive edges and to the other $M/2$ in $B$ via negative edges. The root node $r$ is connected to all bait vertices via positive edges. The function of the bait is to force any optimal burning sequence to ignite $r$ yellow in the first round and set off a timer. In the following, we show that any other strategy would yield a much smaller objective value.

First, note that if $r$ is ignited yellow in the first round, one has all $M$  vertices of $B$ burned yellow by the second round. However, any strategy that ignites a vertex other than $r$ in the first round, regardless of its colour, will burn around $M/2$ vertices of $B$ green in the second round. When $M$ is large enough to satisfy 
$$M \geq  M/2 + \abs{T} + N \Rightarrow M/2 \geq N (k + 1) + 1,$$
the optimal strategy will burn $r$ yellow in the first round (a loss of $M/2$ in objective value cannot be compensated by future ignitions). $M \geq N^3$ certainly satisfies this condition.
    
Next, we show that if we can compute $Y_k(\Gamma)$ exactly in polynomial time, then we can also decide any given instance of the Set Cover problem in polynomial time.

\begin{claim}
$Y_k(\Gamma) \geq M + N \cdot k + n + k + 1$ if and only if the set cover instance is a $\textsf{Yes}$ case, that is, there exists a subset $\cS'$ of size $k$ for which $\bigcup_{s \in \cS'} s = U$.
\end{claim}

\begin{proof}[Proof of Claim]
\renewcommand{\qedsymbol}{$\Box$\ \emph{(Claim.)}}
We distinguish between the $\textsf{Yes}$ and $\textsf{No}$ instances of the Set Cover problem:
\begin{itemize}
\item \textbf{$\textsf{Yes}$ Case.} Assume there exists $\cS' \subset \cS$ of size $k$ for which $\bigcup_{s \in \cS'} s = U$. Consider the strategy where we ignite $r$ yellow in the first round and ignite $\cS$ vertices from $V(H)$ that correspond to the sets in $\cS'$ yellow in rounds $2, 3, \ldots, k + 1$. In round $k + 2$, a green fire burns every remaining vertex from $V_\cS$ as well as those $V_U$ vertices that have not been marked for coverage so far. The number of yellow vertices at the end of this strategy is at least $1 + N \cdot k$ (from $T$) plus $M$ (from $B$) plus $k$ (from $V_\cS$-vertices in $H$ ignited yellow) plus $n$ (all of the $V_U$) since the sets in $\cS'$ will cover the entirety of $U$, for a total of $M + N \cdot k + n + k + 1$.

\item \textbf{$\textsf{No}$ Case.} Assume that no $\cS' \subset \cS$ of size $k$ can cover the set  $U$. Consider an optimal sequence that burns the maximum number of nodes yellow. If multiple such optimal sequences exist, consider one with the least number of vertices ignited  yellow from $V_U$. This sequence ignites $r$ yellow in the first round, as discussed earlier. We argue that in this sequence, at most $k$ vertices from $V_{\cS}$ get to spread their fire before the process ends. To see why, observe that the paths of length $k - 1$ will spread the fire to all vertices of $H$ (except those marked for coverage in previous rounds) by round $k + 2$. Therefore, even if all the nodes after $r$ in the sequence are chosen from $V_{\cS}$, at most $k$ of them get to pass the fire along to $V_U$. Note that no $V_U$ vertex is ignited yellow. The reason is that if a vertex $u \in V_U$ is ignited yellow in the optimum sequence, by optimality we have no vertex $v \in V_{\cS}$ with an edge to $u$ (such a $v$ exists because we assumed no isolated elements) ignited yellow so far. Replacing $u$ in the sequence with $v$ does not decrease the number of yellow vertices, but it decreases the number of $V_U$ vertices in the sequence ignited yellow, which is a contradiction to the choice of the optimal sequence. Next, note that since we have a $\textsf{No}$ instance of the Set Cover problem, the set of $V_{\cS}$ vertices ignited yellow can burn at most $n - 1$ vertices from $V_U$ yellow. Even assuming that all the path vertices are burned yellow, at most $M + N \cdot k + n + k$ vertices ($M$ from $B$, $N \cdot k + 1$ from paths and $r$, $n - 1$ from $V_U$, and $k$ from $V_\cS$) are burned yellow. Note that in this analysis, we do take the vertices of $V_U$ that are marked for coverage in the last round into account by assuming that the optimal sequence burns them all yellow one by one after round $k + 2$. Altogether, it would be at least one less than what is achievable in the $\textsf{Yes}$ case. \qedhere
\end{itemize}
\end{proof}
The gap of at least $1$ in the objective value of the $\textsf{Yes}$ and $\textsf{No}$ cases means that if we can exactly compute $Y_k(\Gamma)$ for any given $\Gamma$ and $k$ in polynomial time, we could also correctly decide the Set Cover problem in polynomial time, which is not possible unless $\textsf{P} = \textsf{NP}$.
\end{proof}
\begin{figure}[t]
    \centering
    \resizebox{\textwidth}{!}{
\begin{tikzpicture}[
  >={Stealth[length=2.6mm]},
  every node/.style={font=\small},
  vtx/.style={circle,draw,fill=white,minimum size=6.5mm,inner sep=0pt},
  tvtx/.style={circle,draw,fill=yellow!14,minimum size=5.5mm,inner sep=0pt,font=\scriptsize},
  svtx/.style={circle,draw,fill=cyan!12,minimum size=5.5mm,inner sep=0pt,font=\scriptsize},
  src/.style={circle,draw,fill=black!10,minimum size=9mm,inner sep=0pt,font=\normalsize},
  bait/.style={circle,draw,fill=black!5,minimum size=6mm,inner sep=0pt,font=\scriptsize},
  pint/.style={circle,draw,fill=white,minimum size=4.4mm,inner sep=0pt,font=\scriptsize},
  pedge/.style={->,green!55!black,thick},
  nedge/.style={->,red!75!black,thick,dashed},
  pedgeu/.style={-,green!55!black,thick},
  nedgeu/.style={-,red!75!black,thick,dashed},
  fanp/.style={->,green!55!black,semithick,opacity=0.85},
  fann/.style={->,red!75!black,semithick,dashed,opacity=0.85},
  pedgelab/.style={green!45!black,font=\scriptsize},
  nedgelab/.style={red!70!black,font=\scriptsize},
  boxlab/.style={font=\footnotesize\bfseries},
]

% =================== TIMER GADGET (left) ===================
\node[src] (r) at (0,0) {$r$};

% path p_1
\node[svtx] (s1) at (1.7, 1.6) {$s_1$};
\node[pint] (m1) at (2.8, 1.6) {};
\node[tvtx] (t1) at (3.9, 1.6) {$t_1$};
\draw[pedge] (r) -- (s1); \draw[pedge] (s1) -- (m1);
\draw[pedge,dotted,line width=0.9pt] (m1) -- (t1);
\node[font=\scriptsize] at (2.8,2.05) {$p_1$};

% path p_2
\node[svtx] (s2) at (1.7, 0.0) {$s_2$};
\node[pint] (m2) at (2.8, 0.0) {};
\node[tvtx] (t2) at (3.9, 0.0) {$t_2$};
\draw[pedge] (r) -- (s2); \draw[pedge] (s2) -- (m2);
\draw[pedge,dotted,line width=0.9pt] (m2) -- (t2);
\node[font=\scriptsize] at (2.8,0.45) {$p_2$};

\node at (2.8,-1.0) {$\vdots$};

% path p_n
\node[svtx] (sn) at (1.7, -2.0) {$s_N$};
\node[pint] (mn) at (2.8, -2.0) {};
\node[tvtx] (tn) at (3.9, -2.0) {$t_N$};
\draw[pedge] (r) -- (sn); \draw[pedge] (sn) -- (mn);
\draw[pedge,dotted,line width=0.9pt] (mn) -- (tn);
\node[font=\scriptsize] at (2.8,-1.55) {$p_N$};

\begin{scope}[on background layer]
  \node[draw,rounded corners=8pt,fill=cyan!3,thick,
        fit=(r)(s1)(t1)(tn),inner sep=12pt] (timerbox) {};
\end{scope}
\node[boxlab,align=center] at ($(timerbox.north)+(0,0.45)$)
  {timer gadget $T$\\[-1pt]{\scriptsize\mdseries $r$ and paths $p_i$ (length $k - 1$)}};

% =================== GRAPH H (center) ===================
\def\vx{6.6}
\node[vtx] (v1) at (\vx, 1.6)  {$v_1$};
\node[vtx] (v2) at (\vx, 0.0)  {$v_2$};
\node at (\vx,-1.0) {$\vdots$};
\node[vtx] (vn) at (\vx, -2.0) {$v_N$};

\draw[-,black!55] (v1) -- (v2);
\draw[-,black!55] (v2) to[bend right=25] (vn);

\begin{scope}[on background layer]
  \node[draw,rounded corners=8pt,fill=black!2,thick,
        fit=(v1)(v2)(vn),inner sep=14pt] (gbox) {};
\end{scope}
\node[boxlab,align=center] at ($(gbox.north)+(0,0.45)$)
  {graph $H$\\[-1pt]{\scriptsize\mdseries vertex set $V=\{v_1,\dots,v_N\}$}};

% =================== BAIT GADGET (right) ===================
\def\px{10.6}
\node[bait] (b1) at (\px, 2.0) {$b_1$};
\node[bait] (b2) at (\px+1.0, 1.0) {$b_2$};
\node[bait] (b3) at (\px+0.9, -0.4) {$b_3$};
\node[bait] (b4) at (\px-0.1, -1.2) {};
\node[bait] (b5) at (\px-1.0, -0.2) {};
\node[bait] (b6) at (\px-0.6, 1.2) {};

\foreach \a/\b/\s in {%
  b1/b2/p, b2/b3/n, b3/b4/p, b4/b5/n, b5/b6/p, b6/b1/n,
  b1/b3/n, b1/b4/p, b2/b5/p, b2/b4/n, b3/b6/p, b5/b1/n, b6/b3/n}{
  \if p\s
    \draw[pedge,opacity=0.45] (\a) -- (\b);
  \else
    \draw[nedge,opacity=0.45] (\a) -- (\b);
  \fi
}

\begin{scope}[on background layer]
  \node[draw,rounded corners=8pt,fill=blue!3,thick,
        fit=(b1)(b2)(b3)(b4)(b5)(b6),inner sep=13pt] (baitbox) {};
\end{scope}
\node[boxlab,align=center] at ($(baitbox.north)+(0,0.72)$)
  {bait gadget $B$\\[-1pt]{\scriptsize\mdseries clique on $M$ vtcs; each $b_i$: deg $M-1$, half $+$/half $-$}};

% =================== INTER-BOX EDGES (representative only) ===================
% r -> B positive (to all): one representative high arc
\draw[pedge] (r) to[out=80,in=150] ($(baitbox.north west)+(0.3,0.0)$)
      to[out=-30,in=140] (b6);
\node[pedgelab] at (2.7,3.7) {$+$: $r\to B$ (all of $B$)};

% t_i -> v_i : negative last edge (representative on p_1 and p_n)
\draw[nedge] (t1) -- (v1);
\draw[nedge] (t2) -- (v2);
\draw[nedge] (tn) -- (vn);
\node[nedgelab] at (5.15,1.95) {$-$};

% path vertices -> B : half +/half - (a few representatives)
\draw[fanp] (t1) to[out=30,in=160] (b6);
\draw[fann] (t2) to[out=25,in=175] (b5);
\draw[fann] (mn) to[out=-25,in=210] (b3);
\draw[fanp] (mn) to[out=-35,in=210] (b4);
\node[align=center,font=\scriptsize] at (3.0,-3.05)
  {every path vertex $\to B$\\[-2pt](half $+$, half $-$)};

% v_i -> B : half +/half - (a few representatives)
\draw[pedge] (v1) to[out=20,in=150] (b1);
\draw[nedge] (vn) to[out=-15,in=225] (b4);
\node[align=center,font=\scriptsize] at (8.7,-2.9)
  {each $v_i\to B$\\[-2pt](half $+$, half $-$)};

% =================== LEGEND ===================
\begin{scope}[shift={(0.4,-4.4)}]
  \draw[pedge] (0,0) -- (0.9,0);
  \node[right=1pt,font=\scriptsize] at (0.9,0) {positive ($+$)};
  \draw[nedge] (3.4,0) -- (4.3,0);
  \node[right=1pt,font=\scriptsize] at (4.3,0) {negative ($-$)};
  \draw[-,black!55] (6.9,0) -- (7.8,0);
  \node[right=1pt,font=\scriptsize] at (7.8,0) {edge of $H$};
\end{scope}
\end{tikzpicture}    
}
\caption{The gadget allows us to simulate a deadline of $k$ in directed settings.}
\label{fig:timer}
\end{figure}
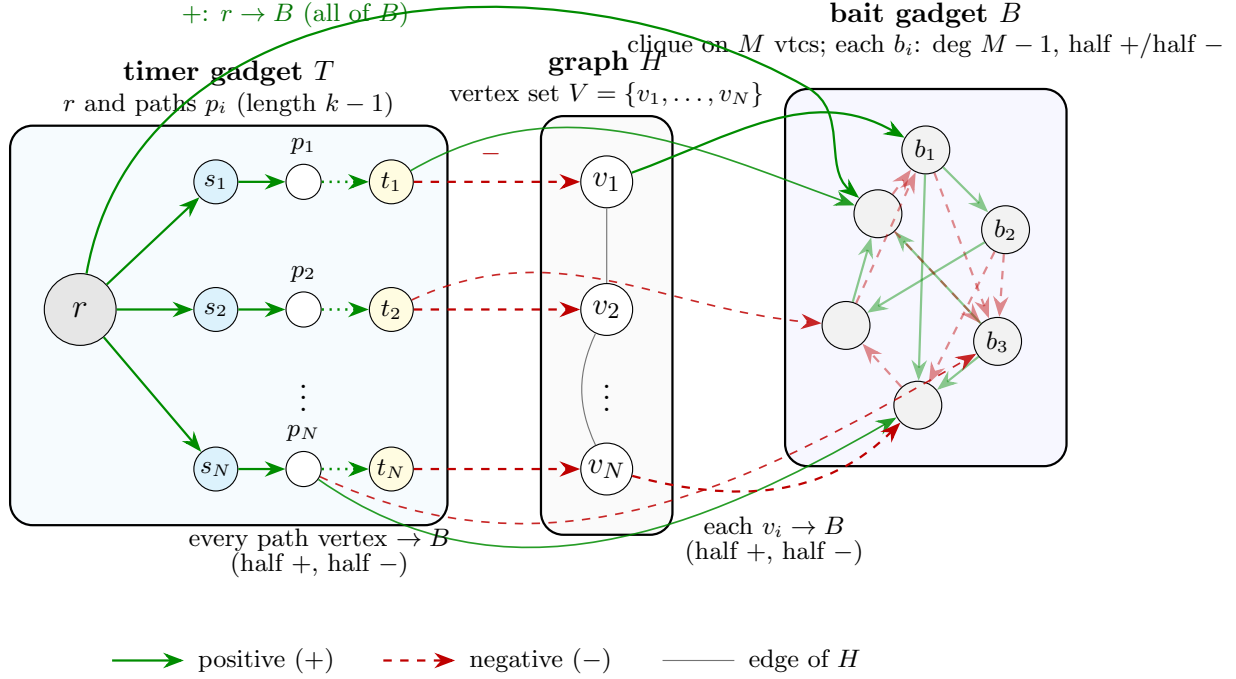

\begin{remark}
    The reduction to the signed graph $\Gamma$ in the proof of Theorem~\ref{thm:DSGBwoD_hardness} serves as a general template of gadget reductions that allow us to extend the hardness results from settings with a deadline to those without in directed or undirected graphs. We use this template in the proof of the following theorem.
\end{remark}

\begin{theorem}\label{thm1:USGBwoD_hardness}
Given an undirected signed graph $\Gamma$ and an integer $\tau$, it is $\textsf{NP}$-complete to
decide whether $Y(\Gamma)\ge\tau$. In particular, unless $\textsf{P}=\textsf{NP}$, there is no
polynomial-time exact algorithm for computing $Y(\Gamma)$, given an undirected signed graph
$\Gamma$ as input.
\end{theorem}

The rest of this section is devoted to the proof. We reduce from the Graph Burning problem. An
instance of that problem is a graph $G$ on $n$ vertices together with a positive integer $m$, and
the question is whether $b(G)\le m$. This problem is $\textsf{NP}$-complete \cite{BessyBJRR17}.
Since $b(G)\le n$ holds for every graph, every instance with $m\ge n$ is a $\textsf{Yes}$ instance,
and we map all of those to one fixed $\textsf{Yes}$ instance of our problem. So we may assume
$1\le m<n$. Let $\Delta$ be the maximum degree of $G$ and put
\[
  k=m+1,\qquad w=\Delta+2k-1,\qquad P=nw(k-1),\qquad L=P+k,\qquad M=2P+2 .
\]
Throughout we use $Y(\Gamma)=\max_{\mathcal{I}\in\mathcal{B}(\Gamma)}
\abs{\mathcal{Y}_{\mathcal{I}}}$, as adopted after Lemma~\ref{lem:symmetry}, and for
$\mathcal{I}\in\mathcal{B}(\Gamma)$ and a vertex $v$ we write $\beta(v)=\min\{t:
\nu_{\mathcal{I},t}(v)\neq 0\}$, as in Section~\ref{sec:paths}. A burning sequence of $\Gamma$ is
written $\mathcal{I}=(z_1^{c_1},\ldots,z_q^{c_q})$, so that $z_s$ is the vertex ignited at time $s$.

The signed graph $\Gamma=\Gamma(G,m)$ is constructed as follows. Start with $G$ and let every edge
of $G$ be positive. For every $v\in V(G)$ add $L$ new vertices and join each of them to $v$ by a
positive edge. These $L$ vertices form the set $Q(v)$ of \emph{pendants} of $v$. Add a new vertex
$r$. For every $v\in V(G)$ and every $i\in[w]$ add a path
$r\,p^{v,i}_1p^{v,i}_2\cdots p^{v,i}_{k-1}\,v$ on $k-1$ new internal vertices, all of whose edges
are positive except $p^{v,i}_{k-1}v$, which is negative. These $w$ paths are the \emph{timer paths}
of $v$. Note that $k\ge 2$, so every timer path has at least one internal vertex and no vertex of $G$
is adjacent to $r$. Let $W$ be the set of all vertices constructed so far other than $r$, so that
$\abs{W}=N:=n(L+1)+P$, and list $W$ as $w_1,\ldots,w_N$. Finally add an independent set
$B=\{b_1,\ldots,b_M\}$ of new vertices, join $r$ to every $b_j$ by a positive edge, and join $w_i$
to $b_j$ by an edge of sign $(-1)^{i+j}$ for all $i\in[N]$ and $j\in[M]$. We call $B$ the
\emph{bait}. Put
\[
  \tau=M+1+n(L+1),
\]
so that $\abs{V(\Gamma)}=\tau+P$. Since $k\le n$ and $w<3n$, the signed graph $\Gamma$ has
$O(n^4)$ vertices and $O(n^7)$ edges, and it is constructed from $(G,m)$ in polynomial time.

Two properties of the bait will be used repeatedly. Since $M$ is even,
\begin{equation}\label{eqn:bait_rows}
  \text{every } x\in W \text{ has exactly } M/2=P+1 \text{ positive and } M/2 \text{ negative
  neighbours in } B .
\end{equation}
Since $N\ge L+1+P=2P+k+1\ge 2P+3$, each of the two parity classes of indices in $[N]$ has at least
$\lfloor N/2\rfloor\ge P+1$ members, so
\begin{equation}\label{eqn:bait_columns}
  \text{every } b\in B \text{ has at least } P+1 \text{ positive and at least } P+1 \text{ negative
  neighbours in } W .
\end{equation}
We write $\Gamma'=\Gamma-B$ and let $d'$ denote the distance in $\Gamma'$.

We first record a standard reformulation of the burning number.

\begin{lemma}\label{lem:burning_cover}
Let $G$ be a graph and let $m\ge1$. Then $b(G)\le m$ if and only if there are vertices
$y_1,\ldots,y_m$ of $G$ with $V(G)=\bigcup_{i=1}^{m}N^{m-i}_G[y_i]$, where
$N^{\rho}_G[y]=\{u\in V(G):d_G(y,u)\le\rho\}$.
\end{lemma}

\begin{proof}
By definition of graph burning.
\end{proof}

The next lemma shows that the bait forces the first ignition.

\begin{lemma}\label{lem:bait}
Let $\mathcal{I}\in\mathcal{B}(\Gamma)$ and suppose that $\mathcal{I}$ does not ignite $r$ in yellow
at time $1$. Then $\abs{\mathcal{Y}_{\mathcal{I}}}\le\tau-1$.
\end{lemma}
\begin{proof}
At the end of time $1$ the vertex $z_1$ is the only burned vertex, and we write $c\in\{\pm1\}$ for
its colour. Suppose $z_1\in W$. At the spread step of time $2$ every $b\in B$ has $z_1$ as its only
burned neighbour, so $b$ is burned with the colour $\sigma(bz_1)c$, and by \eqref{eqn:bait_rows}
exactly $M/2=P+1$ vertices of $B$ are burned green. Suppose $z_1\in B$. At the spread step of time
$2$ every $x\in W$ has $z_1$ as its only burned neighbour, so $x$ is burned with the colour
$\sigma(xz_1)c$, and by \eqref{eqn:bait_columns} at least $P+1$ vertices of $W$ are burned green.
Suppose finally that $z_1=r$ and $c=-1$. Then all $M\ge P+1$ vertices of $B$ are burned green at
the spread step of time $2$. Burned vertices never change their colour, so in each case
$\abs{\mathcal{Y}_{\mathcal{I}}}\le\abs{V(\Gamma)}-(P+1)=\tau-1$.
\end{proof}

From now on we consider burning sequences that ignite $r$ in yellow at time $1$.

\begin{lemma}\label{lem:after_r}
Let $\mathcal{I}\in\mathcal{B}(\Gamma)$ ignite $r$ in yellow at time $1$ and write
$\nu_t=\nu_{\mathcal{I},t}$. At the end of the spread step of time $2$ the burned vertices are
exactly $r$, the vertices of $B$, and the vertices $p^{v,i}_1$ with $v\in V(G)$ and $i\in[w]$, and
all of them are yellow. Moreover $z_s\in W$ for every $s\ge2$, and
$\sum_{b\in B}\sigma(xb)\nu_{t-1}(b)=0$ for every $x\in W$ and every $t\ge3$.
\end{lemma}
\begin{proof}
At the spread step of time $2$ the vertex $r$ is the only burned vertex. Its neighbours are the
vertices of $B$ and the vertices $p^{v,i}_1$, and all the corresponding edges are positive. So each
of them has $h_2=1$ and is burned yellow, while every other vertex has $h_2=0$ and stays unburned.
The ignition step only chooses unburned vertices, and $r$ is burned at time $1$ while every vertex
of $B$ is burned at the spread step of time $2$. Since $V(\Gamma)=W\sqcup B\sqcup\{r\}$, this gives
$z_s\in W$ for every $s\ge2$. For $t\ge3$ every vertex of $B$ is burned yellow at the end of time
$t-1$, so the sum in question equals $\sum_{b\in B}\sigma(xb)$, which vanishes by
\eqref{eqn:bait_rows}.
\end{proof}

\begin{lemma}\label{lem:witness}
Let $\mathcal{I}\in\mathcal{B}(\Gamma)$ ignite $r$ in yellow at time $1$ and let $u\notin B$ be a
burned vertex. Then there is a time $s$ with $d'(z_s,u)\le\beta(u)-s$.
\end{lemma}

\begin{proof}
By Lemma~\ref{lem:after_r} every ignited vertex lies in $W\cup\{r\}$.
\end{proof}

\begin{lemma}\label{lem:distance}
Define $\pi\colon W\to V(G)$ by $\pi(x)=x$ for $x\in V(G)$, and $\pi(x)=v$ if $x\in Q(v)$ or $x$ is
an internal vertex of a timer path of $v$. Then $d'(r,u)=k$ for every $u\in V(G)$. Moreover, if
$x\in W$ and $u\in V(G)$ satisfy $d'(x,u)\le k$, then $d_G(\pi(x),u)\le d'(x,u)$.
\end{lemma}
\begin{proof}
In $\Gamma'$ the neighbours of $r$ are the vertices $p^{v,i}_1$, and the neighbours of an internal
vertex of a timer path lie on that path. So a walk in $\Gamma'$ from $r$ to a vertex of $G$
traverses a whole timer path and has length at least $k$, with equality along a timer path of $u$.
Hence $d'(r,u)=k$. Deleting $r$ from $\Gamma'$ leaves $G$ together with the paths
$p^{v,i}_1\cdots p^{v,i}_{k-1}$ and the pendants, each attached to $G$ at the single vertex $v$ with
$\pi$-value $v$. So a walk from $x$ to $u$ avoiding $r$ passes through $\pi(x)$ and has length at
least $d_G(\pi(x),u)$, while a walk from $x$ to $u$ through $r$ has length at least
$1+d'(r,u)=k+1$. Thus $d'(x,u)\le k$ forces a shortest $x,u$ walk to avoid $r$, and the claim
follows.
\end{proof}

The following is straightforward to see.

\begin{lemma}\label{lem:timer_wave}
Let $\mathcal{I}\in\mathcal{B}(\Gamma)$ ignite $r$ in yellow at time $1$, let $v\in V(G)$, let
$i\in[w]$ and write $p_j=p^{v,i}_j$. Let $1\le t\le k$ and suppose that $v$ is unburned at the end
of time $t-1$ and that no $p_j$ is ignited at a time in $\{2,\ldots,t\}$. Then at the end of time
$t$ the vertices $p_1,\ldots,p_{t-1}$ are burned yellow and the vertices $p_t,\ldots,p_{k-1}$ are
unburned.
\end{lemma}

\begin{proof}
Put $p_0=r$ and $p_k=v$, so that the neighbours of $p_j$ outside $B$ are $p_{j-1}$ and $p_{j+1}$
for $1\le j\le k-1$. We induct on $t$. For $t=1$ nothing but $r$ is burned at the end of time $1$,
which is the assertion. Let $2\le t\le k$. The induction hypothesis applies at $t-1$ and gives that
at the end of time $t-1$ the vertices $p_1,\ldots,p_{t-2}$ are burned yellow and
$p_{t-1},\ldots,p_{k-1}$ are unburned. Consider the spread step of time $t$. The vertices of $B$
are unburned at the end of time $1$, and they contribute $0$ to $h_t$ for $t\ge3$ by
Lemma~\ref{lem:after_r}. So $h_t(p_j)$ is determined by $p_{j-1}$ and $p_{j+1}$. For $j=t-1$ the
vertex $p_{t-2}$ is yellow and joined to $p_{t-1}$ by a positive edge, while $p_t$ is unburned,
where for $t-1=k-1$ we use that $v$ is unburned at the end of time $t-1$. Hence $h_t(p_{t-1})=1$
and $p_{t-1}$ is burned yellow. For $j\ge t$ both $p_{j-1}$ and $p_{j+1}$ are unburned at the end
of time $t-1$, so $h_t(p_j)=0$ and $p_j$ stays unburned at the spread step. It is not ignited at
time $t$ by hypothesis.
\end{proof}

Finally, we show that our reduction mapping maps \textsf{Yes} cases to \textsf{Yes} cases (Lemma~\ref{lem:yes_case}) and \textsf{No} cases to \textsf{No} cases (Lemma~\ref{lem:no_case}).

\begin{lemma}\label{lem:yes_case}
If $b(G)\le m$, then $Y(\Gamma)\ge\tau$.
\end{lemma}
\begin{proof}
Let $y_1,\ldots,y_m$ be as in Lemma~\ref{lem:burning_cover}. Define $\mathcal{I}$ as follows.
Ignite $r$ yellow at time $1$. At a time $s\ge2$, ignite $y_{s-1}$ yellow if $s\le m+1$ and
$y_{s-1}$ is unburned; otherwise ignite an unburned vertex of $G$ yellow if one exists, and an
arbitrary unburned vertex yellow if not. Let $T$ be the first time at the end of which every vertex
of $G$ is burned. If $2\le s<T$, then some vertex of $G$ is unburned at the ignition step of time
$s$, so $z_s\in V(G)$.

We claim that for every $t\le k$, every vertex of $G$ that is burned at the end of time $t$ is
yellow. The case $t=1$ holds because only $r$ is burned then. Let $2\le t\le k$ and let $v\in V(G)$
be unburned at the end of time $t-1$. Then $t-1<T$, so no pendant and no internal vertex of a timer
path is ignited at a time in $\{2,\ldots,t-1\}$. Every $x\in Q(v)$ has $v$ as its only neighbour
outside $B$, and $B$ is unburned at the end of time $1$ and contributes $0$ to $h_{t'}$ for
$t'\ge3$ by Lemma~\ref{lem:after_r}. So every pendant of $v$ is unburned at the end of time $t-1$.
Lemma~\ref{lem:timer_wave}, applied with $t-1$ in place of $t$, shows that $p^{v,i}_{k-1}$ is
unburned at the end of time $t-1$ for every $i\in[w]$, because $t-1\le k-1$. Hence all burned
neighbours of $v$ at the spread step of time $t$ lie in $V(G)$ and are joined to $v$ by positive
edges. By the induction hypothesis they are yellow, so $h_t(v)$ is the number of burned neighbours
of $v$ in $G$. Therefore $v$ is burned yellow at time $t$ if it has a burned neighbour in $G$ at the
end of time $t-1$, and it stays unburned otherwise. As the ignitions are yellow, the claim follows.

The same reasoning shows that the fire spreads inside $G$ at unit speed up to time $k$. The
vertex $y_i$ is burned at the end of time $i+1$ at the latest, so every vertex of
$N^{m-i}_G[y_i]$ is burned at the end of time $(i+1)+(m-i)=k$. By Lemma~\ref{lem:burning_cover},
every vertex of $G$ is burned yellow at the end of time $k$.

Let $v\in V(G)$ and $x\in Q(v)$. If $x$ is ignited, then it is ignited yellow. Otherwise, $v$ is
its only neighbour outside $B$, so $x$ is burned at the spread step of time $\beta(v)+1\ge3$ with
the colour $\nu_{\mathcal{I}}(v)=1$, the edge $xv$ being positive. So every pendant is yellow.
Consequently $r$, the $M$ vertices of $B$, the $n$ vertices of $G$ and the $nL$ pendants are all
yellow, whence $\abs{\mathcal{Y}_{\mathcal{I}}}\ge M+1+n(L+1)=\tau$.
\end{proof}

\begin{lemma}\label{lem:no_case}
If $b(G)>m$, then $Y(\Gamma)\le\tau-1$.
\end{lemma}
\begin{proof}
Let $\mathcal{I}\in\mathcal{B}(\Gamma)$. By Lemma~\ref{lem:bait} we may assume that $\mathcal{I}$
ignites $r$ yellow at time $1$.

Suppose first that every vertex of $G$ is burned at the end of time $k$. Let $u\in V(G)$. By
Lemma~\ref{lem:witness} there is a time $s$ with $d'(z_s,u)\le\beta(u)-s\le k-s$, so $s\le k$. If
$s=1$, then $z_1=r$ and $d'(r,u)\le k-1$, contradicting Lemma~\ref{lem:distance}. So $s\ge2$, and
$z_s\in W$ by Lemma~\ref{lem:after_r}. Hence Lemma~\ref{lem:distance} gives
$d_G(\pi(z_s),u)\le k-s$. Put $y_i=\pi(z_{i+1})$ for $1\le i\le m$, which is defined by
Lemma~\ref{lem:after_r}, taking for $y_i$ an arbitrary vertex of $G$ if $\mathcal{I}$ has no
ignition at time $i+1$. Since $k-s=m-i$ for $s=i+1$, we obtain
$V(G)=\bigcup_{i=1}^{m}N^{m-i}_G[y_i]$, so $b(G)\le m$ by Lemma~\ref{lem:burning_cover}, a
contradiction. 

So let $v\in V(G)$ be unburned at the end of time $k$. Call a timer path of $v$ \emph{touched} if
one of its internal vertices is ignited at a time in $\{2,\ldots,k\}$. Let $a$ be the number of
touched timer paths of $v$ and let $b$ be the number of pendants of $v$ ignited at a time in
$\{2,\ldots,k\}$. Distinct timer paths of $v$ have disjoint sets of internal vertices, and these
are disjoint from $Q(v)$. So each of the $a+b$ objects counted uses up one of the $k-1$ ignitions
at times $2,\ldots,k$, and $a+b\le k-1$.

Let the $i$-th timer path of $v$ be untouched. Lemma~\ref{lem:timer_wave}, applied with $t=k$,
gives that $p^{v,i}_{k-1}$ is burned yellow at the end of time $k$. The edge $p^{v,i}_{k-1}v$ is
negative, so this vertex contributes $-1$ to $h_{k+1}(v)$. The neighbours of $v$ are its at most
$\Delta$ neighbours in $G$, its $L$ pendants, the $w$ ends of its timer paths, and the vertices of
$B$. The vertices of $B$ contribute $0$ by Lemma~\ref{lem:after_r}. A pendant of $v$ has $v$ as its
only neighbour outside $B$, so it is burned at the end of time $k$ only if it is ignited at a time
in $\{2,\ldots,k\}$. Every contribution other than those of the $w-a$ untouched timer paths is at
most $1$. Hence
\[
  h_{k+1}(v)\;\le\;\Delta+b+a-(w-a)\;=\;\Delta+b+2a-w\;\le\;\Delta+2(k-1)-w\;=\;-1 ,
\]
and $v$ is burned green at the spread step of time $k+1$.

At most $k$ pendants of $v$ are ignited at times $2,\ldots,k+1$. Let $x\in Q(v)$ be a pendant that
is not one of them. At the spread step of time $k+1$ the vertex $x$ has no burned neighbour outside
$B$, since $v$ is unburned at the end of time $k$, so $x$ is unburned at the end of time $k+1$. At
the spread step of time $k+2$ its only burned neighbour outside $B$ is $v$, which is green, and the
edge $xv$ is positive. So $x$ is burned green. Hence at least $L-k$ pendants of $v$ are green, and
together with $v$ this gives
\[
  \abs{\mathcal{Y}_{\mathcal{I}}}\;\le\;\abs{V(\Gamma)}-(L-k+1)\;=\;\tau+P-(P+1)\;=\;\tau-1 .
\]
As $\mathcal{I}$ was arbitrary, $Y(\Gamma)\le\tau-1$.
\end{proof}

\begin{figure}[t]
    \centering
\begin{tikzpicture}[
  every node/.style={font=\small},
  gvtx/.style={circle,draw,fill=white,minimum size=6mm,inner sep=0pt,font=\scriptsize},
  pend/.style={circle,draw,fill=green!8,minimum size=4.2mm,inner sep=0pt,font=\tiny},
  tint/.style={circle,draw,fill=white,minimum size=4.0mm,inner sep=0pt,font=\tiny},
  src/.style={circle,draw,fill=black!12,minimum size=8.5mm,inner sep=0pt,font=\normalsize},
  bait/.style={circle,draw,fill=black!5,minimum size=5.4mm,inner sep=0pt,font=\scriptsize},
  pedge/.style={-,green!55!black,thick},
  nedge/.style={-,red!75!black,thick,dashed},
  pedgelab/.style={green!45!black,font=\scriptsize},
  nedgelab/.style={red!70!black,font=\scriptsize},
  boxlab/.style={font=\footnotesize\bfseries},
]

% ---------- ROOT ----------
\node[src] (r) at (0,0) {$r$};

% ---------- GRAPH G (two representative vertices u, v with an edge) ----------
\def\gx{7.8}
\node[gvtx] (v) at (\gx, 1.4) {$v$};
\node[gvtx] (u) at (\gx, -1.4) {$u$};
\draw[pedge] (v) -- (u);   % edge of G, positive
\node[pedgelab] at (\gx+0.35,0.0) {$+$};

\begin{scope}[on background layer]
  \node[draw,rounded corners=7pt,fill=black!2,thick,fit=(v)(u),inner sep=13pt] (gbox) {};
\end{scope}
\node[boxlab,align=center] at ($(gbox.north)+(0,0.4)$)
  {graph $G$ (edges $+$)};

% ---------- TIMER PATHS for v (w of them; show 2 + dots) ----------
% each path: r - p1 - ... - p_{k-1} - v, all + except last edge (p_{k-1} v) negative.
% show k-1 = 3 internal vertices schematically (dotted middle).
\newcommand{\timerpath}[5]{% #1=yshift #2=name-prefix #3=target(v/u) #4=label side #5=pathlabel
  \node[tint] (#2a) at (2.0,#1) {};
  \node[tint] (#2b) at (3.4,#1) {};
  \node[tint] (#2c) at (5.0,#1) {};
  \draw[pedge] (r) -- (#2a);
  \draw[pedge] (#2a) -- (#2b);
  \draw[pedge,dotted,line width=0.9pt] (#2b) -- (#2c);
  \draw[nedge] (#2c) -- (#3);
}
% two timer paths to v
\timerpath{2.5}{tva}{v}{}{}
\timerpath{1.7}{tvb}{v}{}{}
\node at (3.4,1.15) {$\vdots$};
% one more to v (bottom of its bundle) to suggest w of them
\node[tint] (tvc a) at (2.0,0.55) {};
\node[tint] (tvc b) at (3.4,0.55) {};
\node[tint] (tvc c) at (5.0,0.55) {};
\draw[pedge] (r) -- (tvc a); \draw[pedge] (tvc a) -- (tvc b);
\draw[pedge,dotted,line width=0.9pt] (tvc b) -- (tvc c);
\draw[nedge] (tvc c) -- (v);

% two timer paths to u
\timerpath{-0.55}{tua}{u}{}{}
\timerpath{-1.7}{tub}{u}{}{}
\node at (3.4,-1.15) {$\vdots$};

% brace: w timer paths, length k
\draw[decorate,decoration={brace,amplitude=5pt},thin]
  (1.7,2.9) -- (5.2,2.9)
  node[midway,above=4pt,align=center]{\scriptsize $w=\Delta+2k-1$ timer paths per vertex;\\[-1pt]\scriptsize each length $k$ ($k-1$ internal vtcs), all $+$ except last edge};
\node[nedgelab] at (5.75,1.15) {$-$};
\node[nedgelab] at (5.75,-1.75) {$-$};
\node[pedgelab] at (1.0,2.05) {$+$};

% ---------- PENDANTS Q(v), Q(u) ----------
\node[pend] (qv1) at (\gx+1.5, 2.3) {};
\node[pend] (qv2) at (\gx+1.9, 1.6) {};
\node[pend] (qv3) at (\gx+1.9, 0.8) {};
\node at (\gx+2.25,1.55) {$\cdots$};
\draw[pedge] (v)--(qv1); \draw[pedge] (v)--(qv2); \draw[pedge] (v)--(qv3);
\node[pend] (qu1) at (\gx+1.5, -2.3) {};
\node[pend] (qu2) at (\gx+1.9, -1.6) {};
\node[pend] (qu3) at (\gx+1.9, -0.8) {};
\node at (\gx+2.25,-1.55) {$\cdots$};
\draw[pedge] (u)--(qu1); \draw[pedge] (u)--(qu2); \draw[pedge] (u)--(qu3);
\node[pedgelab,align=center] at (\gx+2.55,2.55) {$L$ pendants\\[-2pt]$Q(v)$ ($+$)};

% ---------- BAIT B ----------
\def\bx{1.2}
\node[bait] (b1) at (\bx,-2.7) {$b_1$};
\node[bait] (b2) at (\bx+1.0,-3.1) {$b_2$};
\node[bait] (b3) at (\bx+2.0,-2.8) {$\cdots$};
\node[bait] (b4) at (\bx+2.9,-3.15) {$b_M$};
\begin{scope}[on background layer]
  \node[draw,rounded corners=7pt,fill=blue!3,thick,fit=(b1)(b2)(b4),inner sep=8pt] (baitbox) {};
\end{scope}
\node[boxlab,align=center] at ($(baitbox.south)+(0,-0.3)$)
  {bait $B$: $M=2P+2$ independent vtcs;\ $r$--$B$ all $+$;\ each $w_i$--$b_j$ sign $(-1)^{i+j}$};

% r -- B positive (representative)
\draw[pedge] (r) to[out=-90,in=150] (b1);
\node[pedgelab] at (0.35,-1.6) {$+$};
% W -- B mixed (representative: one + one -)
\draw[pedge,opacity=0.6] (v) to[out=-140,in=60] (b4);
\draw[nedge,opacity=0.6] (u) to[out=180,in=35] (b3);

% ---------- LEGEND ----------
\begin{scope}[shift={(-0.3,-4.55)}]
  \draw[pedge] (0,0) -- (0.8,0);
  \node[right=1pt,font=\scriptsize] at (0.8,0) {positive ($+$)};
  \draw[nedge] (3.1,0) -- (3.9,0);
  \node[right=1pt,font=\scriptsize] at (3.9,0) {negative ($-$)};
  \draw[pedge,dotted,line width=0.9pt] (6.4,0) -- (7.2,0);
  \node[right=1pt,font=\scriptsize] at (7.2,0) {timer-path interior};
\end{scope}

\end{tikzpicture}
    \caption{The undirected signed graph $\Gamma$ has the property that any exact calculation of $Y(\Gamma)$ encodes a solution to the burning number of the given graph~$G$.}
    \label{fig:ultimate}
\end{figure}

\begin{proof}[Proof of Theorem~\ref{thm1:USGBwoD_hardness}]
A burning sequence of $\Gamma$ has at most $\abs{V(\Gamma)}$ entries, and the burning process it
determines is simulated in polynomial time. So a burning sequence $\mathcal{I}$ with
$\abs{\mathcal{Y}_{\mathcal{I}}}\ge\tau$ is a certificate, and the problem lies in $\textsf{NP}$.

For the hardness, map an instance $(G,m)$ of Graph Burning with $1\le m<n$ to the instance
$(\Gamma(G,m),\tau)$, and map every instance with $m\ge n$ to a fixed $\textsf{Yes}$ instance. The
construction takes polynomial time. By Lemmas~\ref{lem:yes_case} and \ref{lem:no_case} we have
$Y(\Gamma)\ge\tau$ if and only if $b(G)\le m$. Since Graph Burning is $\textsf{NP}$-complete
\cite{BessyBJRR17}, deciding whether $Y(\Gamma)\ge\tau$ is $\textsf{NP}$-complete as well. Finally,
a polynomial-time algorithm computing $Y(\Gamma)$ exactly would decide Graph Burning in polynomial
time, which is impossible unless $\textsf{P}=\textsf{NP}$.
\end{proof}

Undirected signed burning without deadlines sits at the bottom of the hierarchy in
Figure~\ref{fig:burning-hierarchy}, and the two reductions leaving it preserve the objective value.
So the hardness of the two problems above it follows from
Theorem~\ref{thm1:USGBwoD_hardness}.

\begin{corollary}\label{thm1:USGBwD_hardness}
Unless $\textsf{P} = \textsf{NP}$, there is no polynomial-time exact algorithm for computing
$Y_k(\Gamma)$, given an undirected signed graph $\Gamma$ and an integer $k$ as input.
\end{corollary}
\begin{proof}
Let $\Gamma$ be an undirected signed graph on $n$ vertices and set $k=n$. Every burning sequence of
$\Gamma$ has at most $n$ entries, so $\mathcal{B}_{\le n}(\Gamma)=\mathcal{B}(\Gamma)$ and
$Y_n(\Gamma)=Y(\Gamma)$. Now apply Theorem~\ref{thm1:USGBwoD_hardness}.
\end{proof}

\begin{corollary}\label{thm1:DSGBwoD_hardness}
Unless $\textsf{P} = \textsf{NP}$, there is no polynomial-time exact algorithm for computing
$Y(\Gamma)$, given a directed signed graph $\Gamma$ as input.
\end{corollary}
\begin{proof}
Let $\Gamma=(G,\sigma)$ be an undirected signed graph, and let $D$ be the digraph obtained from $G$
by replacing every edge $uv$ by the two arcs $(u,v)$ and $(v,u)$, both of sign $\sigma(uv)$. Then
$N^{+}_{\textsf{in}}(v)$ and $N^{-}_{\textsf{in}}(v)$ in $(D,\sigma)$ coincide with $N^{+}(v)$ and
$N^{-}(v)$ in $\Gamma$ for every vertex $v$. So the two burning processes agree step by step, and
$Y((D,\sigma))=Y(\Gamma)$. Now apply Theorem~\ref{thm1:USGBwoD_hardness}.
\end{proof}

\section{Conclusions}
Motivated by the notion of social networks with positive and negative ties, this paper introduced signed graph burning with two colours, as well as a related parameter called the \emph{plurality number}.

We derived exact values for the plurality number for special classes of signed paths, while showing that computing or approximating it in general (in certain variants of the problem) is $\textsf{NP}$-hard. It would be interesting to extend both these types of results, for example, by finding upper and lower bounds on the plurality number of classes of signed trees, as well as by identifying non-trivial subproblems of our $\textsf{NP}$-hard problems that are solvable in polynomial time.

Finally, exploring connections between signed burning and covering radii of codes might lead to new coding-theoretic tools applicable in graph burning.

\section{Acknowledgements}

S.\ Fallat was supported in part by an NSERC Discovery Grant, application no.\ RGPIN-2025-05272.
K.\ Khodamoradi was supported through the NSERC Discovery Grants program, application no.\
RGPIN-2024-06360.
S.\ Zilles was supported through a Canada CIFAR AI Chair at the Alberta Machine Intelligence Institute (Amii), through an NSERC Canada Research Chair, through the New Frontiers in Research Fund under grant no.\ NFRFE-2023-00109 and through the NSERC Discovery Grants program under application no.\ RGPIN-2026-05856.

\section{AI Disclosure}
The proof of Theorem \ref{thm1:USGBwoD_hardness} as well as some parts of the proof of Theorems \ref{thm:path-lower} and \ref{thm:upper} were produced in collaboration with Claude Opus 5.5. All mathematical arguments and results were independently verified by the authors, who take full responsibility for the content of the paper. All other proofs, as well as the signed burning model and related parameters are devised by the four authors.

\end{document}